\documentclass[12pt]{amsart}

\usepackage{amsfonts , graphicx, amsmath, amssymb, color}
 \usepackage[applemac]{inputenc}

\usepackage[all, 2cell]{xypic}
\usepackage[margin = 1in]{geometry}

\newcommand \pa {\partial}
\newcommand \p {\partial}
\newcommand \lra {\longrightarrow}

\newcommand \pp {\mathfrak{p}}
\newcommand \I {\iota}

\newcommand \absv [1]{\left \lvert #1 \right \rvert }

\newcommand \lp {\left(}
\newcommand \rp {\right)}
\newcommand{\set}[1]{\left\{ #1 \right\} }
\newcommand{\wt}[1]{\widetilde{#1}}
\newcommand{\ov}[1]{\overline{#1}}

\newtheorem{theorem}{Theorem}
\newtheorem{corollary}[theorem]{Corollary}

\newtheorem{lemma}[theorem]{Lemma}
\newtheorem{example}[theorem]{Example}

\newtheorem{proposition}[theorem]{Proposition}
\newtheorem{remark}[theorem]{Remark}
\newtheorem*{acknowledgements}{Acknowledgements}
\newtheorem{definition}[theorem]{Definition}
\newcommand\cA{\mathcal{A}}
\newcommand\cF{\mathcal{F}}
\renewcommand\Im{\operatorname{Im}}

\newcommand\phg{\operatorname{phg}}
\newcommand\bbN{\mathbb{N}}
\newcommand\tU{\widetilde{\mathcal{U}}}
\newcommand\cU{\mathcal{U}}
\newcommand\tV{\widetilde{V}}
\newcommand\tL{\widetilde{L}}
\newcommand\tF{\widetilde{F}}
\newcommand\pr{\operatorname{pr}}
\newcommand\bbB{\mathbb{B}}
\newcommand\bbR{\mathbb{R}}
\newcommand\bbS{\mathbb{S}}
\newcommand\bbZ{\mathbb{Z}}
\newcommand\fob{\mathcal{F}}
\newcommand\foc{\mathcal{F}-c}
\newcommand\cf{cf\@. }

\newcommand\SO{\operatorname{SO}}

\title[Hodge cohomology of foliated boundary metrics]{Hodge cohomology of some foliated boundary and foliated cusp metrics}
\author{Jesse Gell-Redman}
\address{Department of Mathematics, University of Toronto}
\email{jgell@math.toronto.edu}
\author{Fr\'ed\'eric Rochon}
\address{Département de Mathématiques, UQÀM}
\email{rochon.frederic@uqam.ca}

\begin{document}

\maketitle

\begin{abstract}
For fibred boundary and fibred cusp metrics,  Hausel, Hunsicker, and Mazzeo identified the
space of $L^2$ harmonic forms of fixed degree with the images of maps between intersection
cohomology groups of an associated stratified space obtained by
collapsing the fibres of the fibration at infinity onto its base.  In
the present paper, we obtain a generalization of this result to
situations where, rather than a fibration at infinity, there is a Riemannian foliation
with compact leaves admitting a resolution by a fibration.   If the
associated stratified space (obtained now by collapsing the leaves of
the foliation) is a Witt space and if the metric considered is a foliated cusp metric, then no such resolution is required.

\end{abstract}

\tableofcontents

\section{Introduction}

The Hodge theorem states that for a complete, compact Riemannian
manifold without boundary, the
space of $L^{2}$ harmonic forms -- the `Hodge cohomology' -- is
isomorphic to the de Rham cohomology.  For manifolds that are either not
complete or not compact, no general relationship between the Hodge
cohomology and a topological invariant is known, but there is a wealth of Hodge type theorems in various
settings: on manifolds with cylindrical ends \cite{APSI}, on singular
algebraic varieties
\cite{CGM1982}, on locally symmetric spaces \cite{Z1982},
\cite{SS1990}, on asymptotically geometrically finite hyperbolic
quotients \cite{Ma1988}, \cite{MP1990}, and the well-known work of
Cheeger \cite{C1980} and Nagase \cite{Nagase1987} (see also \cite{BHS1992}), which relates the Hodge cohomology of
manifolds
with iterated conical singularities with the
intersection cohomology groups of Goresky and Macpherson,
\cite{GM1980}, \cite{GM1983}.  
These same intersection cohomology groups appear in the work of
Hausel, Hunsicker, and Mazzeo \cite{HHM2004} on the Hodge cohomology of
fibred boundary and
fibred cusp metrics, two natural geometries defined on a
smooth manifold whose boundary  is diffeomorphic to
a fibration.  

Our goal in this paper is to extend the results of
Hausel, Hunsicker, and Mazzeo to the
more general case where the boundary is diffeomorphic to a Seifert
fibration.  A Seifert fibration is, loosely speaking, a 
foliation whose space of leaves is an orbifold. See
Section \ref{sfs.0} for a detailed description.  On such manifolds, fibred
boundary and fibred cusp metrics have natural analogues:
\textit{foliated} boundary and
\textit{foliated} cusp metrics as introduced in \cite{pomfb}, and it
is these metrics whose Hodge cohomology we study here.  To keep the analysis part of the paper tractable, we will
further assume that the Seifert fibration is
\textit{good}, meaning that the space of leaves of the boundary foliation is the 
quotient of a smooth compact manifold by a smooth, properly discontinuous action of
a finite group.  This condition is needed only in Section \ref{wtl.0},
and we hope to remove it in a future work.

Thus, let $M$ be a non-compact
manifold inside a compact, smooth manifold with boundary $\ov{M},$ whose boundary $\p
M = \ov{M} - M$ is the total space of a Seifert fibration $\cF$, for the moment not necessarily good.  
Let $B$ denote the space of leaves of $\mathcal{F}$,
and let
\begin{equation}\label{eq:B}
  \pi \colon \p M \lra B
\end{equation}
be the associated projection.    Let $x$ be a boundary defining function (b.d.f.) for $\p M$, i.e.\ 
$x \in C^{\infty}(\ov{M})$ with  $x^{-1}(0) = \p M$, $dx \rvert_{\p M} \neq 0$ and $x>0$ on $M$.  For
small $\epsilon > 0$, the set $x^{-1}([0,\epsilon))$ is diffeomorphic
to $\p M \times [0,\epsilon)_{x}$, and we extend the projection $\pi$ to
this neighborhood of the boundary in the obvious way. 
Then an exact \textbf{foliated boundary metric} is a Riemannian metric which on
$\p M \times [0,\epsilon)_{x}$ takes the form
\begin{equation}
  \label{eq:foliatedboundarymet}
  g_{\fob} = \frac{dx^{2}}{x^{4}} + \frac{\pi^*h}{x^{2}} + k,
\end{equation}  
where $h$ is an orbifold Riemannian metric on $B$ and $k$ is a $(0,2)$-tensor
which restricts to a Riemannian metric on each leaf of the foliation $\cF$.  
Similarly, an exact \textbf{foliated cusp metric} is a metric of the form
\begin{equation}
  \label{eq:foliatedcuspmet}
   g_{\foc} := \frac{dx^{2}}{x^{2}} + \pi^{*} h + x^{2} k,
\end{equation}
with $h$ and $k$ as above.  General (i.e.\ non-exact) foliated metrics
are permitted to have ``cross-terms''; see Section
\ref{sec:weighttoint}. 
When the base $B$ of  the Seifert fibration is smooth, $\cF$
corresponds to the fibres of a smooth fibration, and we recover from
these definitions the notion of exact fibred boundary metric and
exact fibred cusp metric considered for instance in \cite{HHM2004}.

For a choice of foliated boundary or foliated cusp metric $g$ on
$M$, let $L^2\mathcal{H}^k(M,g)$
denote the space of $L^{2}$ harmonic forms of degree $k$, i.e.\ those  $\omega \in \Omega^{k}(M)$
with $\left \| \omega \right \|_{L^{2}(g)} < \infty$ solving $d \omega = 0$ and $\delta
\omega = 0$, where $\delta$ is the adjoint of $d$ with respect to the
$L^{2}$ norm induced by $g$.  To relate $L^{2}\mathcal{H}^{k}(M, g)$ with some
topological data, we 
consider, instead of $M$, the space $X$
obtained by collapsing the leaves of the foliation on $\p M$ onto the space of leaves $B$.  To be precise, let
\begin{equation}
  \label{eq:X}
  X := \ov{M} / \sim, \mbox{ where } p \sim q \iff   p=q \; \mbox{or}
  \; p,q \in \p M \mbox{ with } \pi(p) =
  \pi(q).
\end{equation}
There is a corresponding collapsing map $c_{\pi}: \overline{M}\to X$
which is the identity on $M$ and is given by the projection $\pi$ on
$\pa M$.  We identify $B$ with $c_{\pi}(\p M) \subset X$.
The space $X$ is a (smoothly) stratified
space (see for instance \cite{ALMP2011} or \cite{DLR2011} for a
definition), and as such, the intersection homology and cohomology groups of
Goresky and MacPherson \cite{GM1980} can be defined thereon.  These
groups are not homotopically invariant like singular homology and cohomology
groups, but they are topological invariants.
They are defined in terms of a perversity
function, i.e.\ a map $\pp \colon \set{0, 1, \dots, n} \lra \mathbb{N}$
satisfying $\pp(\ell) \le \pp(\ell + 1) \le \pp(\ell) + 1$, and a
stratification, i.e.\ is a nested sequence
\begin{equation}\label{eq:strata}
X \supset X_{n - 2}
\supset \dots \supset X_{j} \supset \dots \supset X_{0},
\end{equation}
where
$X - X_{n - 2}$ is a smooth manifold and $X_{j} - X_{j - 1}$ is either
empty or a manifold of dimension $j$.  The group $I\!H^{k}_{\pp}(X)$ is the $k^{th}$ cohomology group of
the complex of cochains defined on chains which intersect each stratum
of codimension $\ell$ in a set of dimension at most $k - \ell +
\pp(\ell)$.  
The original intersection homology theory was developed for \textbf{standard}
perversities, i.e.\
those satisfying the additional condition $\pp(0) = \pp(1) = 0$, but
general perversities are now in common use, see e.g.\ \cite{HHM2004},
\cite{N1999}.  For standard perversities, the group
$I\!H^{k}_{\pp}(X)$ turns out to be independent of the
stratification.  For more on intersection cohomology, see \cite{B1983}.

Let
\begin{equation}
  \label{eq:dimensions}
  \begin{split}
    n &:= \dim M \\
    b &:= \dim B  \\
    f &:= n - b - 1 ,
  \end{split}
\end{equation}
so $f$ is the dimension of a typical leaf of $\mathcal{F}$.
As we discuss in Section \ref{sfs.0}, $X$ carries a
natural stratification induced by the orbifold structure of the space of leaves, $B$.  In particular,
the spaces $X_{n - 2}, \dots, X_{n - (f + 1)}$ are all equal to $B$.
In Corollary~\ref{icp.1} below, \textit{we show that the
intersection cohomology of $X$ depends only on 
$\pp(f+1)$.}  Thus, the following definition makes sense.
For $j\in \bbN$, let
\begin{equation}
  \label{eq:inthom1}
  I\!H^{k}_{j}(X, B) := \left\{
    \begin{array}{ccc}
      H^{k}(X - B) &\mbox{ if } & j \leq - 1\\
      I\!H^{k}_{\pp}(X) \mbox{ where }
  \pp(f + 1) = j. &\mbox{ if } & 0 \leq j \leq f - 1 \\
      H^{k}(X,B) &\mbox{ if } & f \leq j,
    \end{array}
    \right.
\end{equation}
c.f.\ Section 2.2, equation (9) of \cite{HHM2004}.

\begin{theorem}\label{thm:main1}
  Let $\ov{M}$ be a manifold with boundary, $\p M = \ov{M} - M$.  Let $\cF$ be a good Seifert fibration on $\pa M$, and let $g_{\fob}$ be a foliated boundary metric
  on $M$. Then for any degree $0 \leq k \leq n$, there are natural
  isomorphisms
  \[
  L^2\mathcal{H}^k(M,g_{\fob}) \longrightarrow \left\{
    \begin{array}{ll}
      \mbox{\rm Im} \, \big(I\!H^k_{f+\frac{b+1}{2} -k}(X,B)
      \longrightarrow I\!H^k_{f+\frac{b-1}{2} - k}(X,B) \big) \qquad
      & \mbox{$b$ odd}  \\
     I\!H^k_{f+\frac{b}{2} -k}(X,B) & \mbox{$b$ even},
    \end{array} \right.
  \]
with $B$ as in \eqref{eq:B} and $X$ as in \eqref{eq:X}.
\end{theorem}
\begin{theorem}\label{thm:main2}
 Notation and assumptions as in Theorem \ref{thm:main1}, let $g_{\foc}$ be a foliated cusp metric
  on $M$. Then for $0 \leq k \leq n$, there is a natural isomorphism
  \[
  L^2\mathcal{H}^k(M,g_{\foc}) \longrightarrow \mbox{\rm Im} \, \left(
    I\!H^k_{\underline{\mathfrak{m}}}(X,B) \longrightarrow
    I\!H^k_{\overline{\mathfrak{m}}}(X,B) \right)
  \]
  where $\underline{\mathfrak{m}}$ and $\overline{\mathfrak{m}}$ are the lower
  middle and upper middle perversities, (see Section \ref{sec:proofs}.) \end{theorem}
As in \cite{HHM2004}, when $X$ is a Witt space (again, see Section \ref{sec:proofs}), the result simplifies, since $I\!H^k_{\underline{\mathfrak{m}}}(X,B) = I\!H^k_{\overline{\mathfrak{m}}}(X,B)$, and the proof involves softer analytical methods.  In particular, it does not requires the Seifert fibration $\cF$ to be good.  This gives the following.    
\begin{theorem}
Let $\ov{M}$ be a manifold with boundary, $\p M = \ov{M} - M$.  Let $\cF$ be a Seifert fibration on $\pa M$, and let $g_{\foc}$ be a foliated cusp metric
  on $M$.  If the space $X$ described above is Witt, then there are natural isomorphisms   
\[
  L^2\mathcal{H}^*(M,g_{\foc})\cong H^*_{(2)}(M,g_{\foc}) \cong I\!H^*_{\underline{\mathfrak{m}}}(X,B),
\]
where $H^*_{(2)}(M,g_{\foc})$ is the $L^2$-cohomology of $(M,g_{\foc})$.
\label{Witt.1}\end{theorem}

When the boundary foliation is induced by a fibration, 
these theorems reduce to the results of \cite{HHM2004}.  Moreover, Theorems \ref{thm:main1} and \ref{thm:main2}
simplify if the typical leaf is a quotient of a
sphere by the free, properly discontinuous action of a finite group, for
then $X$ itself is an orbifold.  Since the intersection cohomology of an orbifold is isomorphic to the singular cohomology for any standard perversity (a well-known fact, see Theorem~\ref{ico.1} below), this leads to the following analogues of
Corollaries 1 and 2 from \cite{HHM2004}.

\begin{corollary}\label{thm:fbsphere}
 Let $(M,g_{\fob})$ satisfy the assumptions of Theorem \ref{thm:main1}
 and assume that the
  typical leaf is the quotient of a sphere by the free, properly discontinuous action of a finite group.  Then for any degree 
$0 \leq k \leq n$, there are natural isomorphisms
\begin{eqnarray*}
L^2\mathcal{H}^k(M,g_{\fob})
\cong \left\{ \begin{array}{ll}
H^k(X,B) & k \leq \frac{b-1}{2} \\
\mbox{\rm Im}\, \big(H^k(X,B) \longrightarrow H^k(X) \big)
\quad & k =
\frac{b-1}{2} +1 \\
H^k(X) & \frac{b+1}{2} <k < n-\frac{b+1}{2} \\
\mbox{\rm Im}\,\big(H^k(X) \longrightarrow H^k(X \setminus B)) & k = 
n-\frac{b+1}{2} \\
H^k(X \setminus B) & k\geq n-\frac{b-1}{2}
\end{array} \right.
\label{sphereodd}
\end{eqnarray*}
if $b$ is odd, and
\begin{eqnarray*}
L^2\mathcal{H}^k(M,g_{\fob}) 
\cong \left\{ \begin{array}{ll}
H^k(X, B) & k \leq \frac{b}{2} \\
H^k(X) & \frac{b}{2} <k < n-\frac{b}{2} \\
H^k(X \setminus B) & k\geq n-\frac{b}{2}
\end{array} \right.
\label{sphereeven}
\end{eqnarray*}
if $b$ is even.
\end{corollary}

\begin{corollary}\label{thm:fcsphere}
Let $(M,g_{\foc})$ satisfy the assumptions of Theorem \ref{Witt.1}
 and assume that the
  typical leaf is the quotient of a sphere by the free properly discontinuous action of a finite group. Then for any degree 
$0 \leq k \leq n$, 
\[
L^2\mathcal{H}^k(M,g_{\foc}) \cong H^*_{(2)}(M,g_{\foc}) \cong H^k(X).
\]
\end{corollary}

The structure of the present paper is closely related to that of
\cite{HHM2004}.  There, the authors begin by identifying the
intersection cohomology groups with the cohomology of a complex of
weighted differential forms.  They then equate the Hodge cohomology with the images of
maps of certain weighted cohomology groups.  Combining these two
results gives Theorems \ref{thm:main1} and \ref{thm:main2} in the
fibration case.  The lion's share of the work, and all of the
analysis, is in the second step.  Here, the reverse is true; most of
the work lies in identifying intersection
cohomology with weighted cohomology, whereas the identification of
Hodge cohomology with (images of maps of)
weighted cohomology follows directly from the arguments in \cite{HHM2004},
which we use essentially as a black box.

\begin{acknowledgements}
We are grateful to Markus Banagl and Eugénie Hunsicker for some helpful correspondence. 
\end{acknowledgements}

\section{Seifert fibrations and stratified spaces}\label{sfs.0}

As in \cite{Goette2011}, by ``Seifert fibration'', we mean the natural generalization of Seifert
fibrations on 3-manifolds to arbitrary dimensions.  A Seifert
fibration on a compact manifold $M$ is, simply put, a Riemannian foliation $\cF$ with compact leaves.  Since the space of leaves has naturally the structure of an effective orbifold, see for instance \cite{Goette2011} or \cite{Molino1988},    
there is an alternative definition involving orbibundles, which
we describe now
following \cite{Goette2011}.  First, we recall the definition of an orbifold (see for instance
\cite{Thurston}).  

\begin{definition}
An $n$-dimensional smooth \textbf{orbifold} is a second countable Hausdorff space $B$ such that:  
\begin{enumerate}
\item[(i)]  $B$ has a covering by a collection of open sets $\{\cU_i\}$ closed under finite intersections.  For each $\cU_i$, there is a finite group $\Gamma_i$ with a smooth action (on the right) on an open subset $\tU_i\subset \bbR^n$ and a homeomorphism $\varphi_i\colon \tU_i/\Gamma_i\to \cU_i$. 
\item[(ii)]  Whenever $\cU_i\subset \cU_j$, there is an injective
  homomorphism $\nu_{ij}\colon \Gamma_i\to\Gamma_j$ and an embedding
  $\widetilde{\varphi}_{ij}\colon\tU_i\to \tU_j$ equivariant with respect
  to $\nu_{ij}$ and making the following diagram commute,
\[
\xymatrix{
      \tU_i \ar[d]^{\widetilde{\varphi}_{ij}} \ar[r] &  \tU_i/\Gamma_i  \ar[d]^{\varphi_{ij}=\widetilde{\varphi}_{ij}/\Gamma_i  }  & &  \ar[ll]  U_i  \ar@{^{(}->}[d]  \\
      \tU_j \ar[r]  & \tU_j /\Gamma_i \ar[r]^{\nu_{ij}} & \tU_j/ \Gamma_j & \ar[l]_{\varphi_j} \cU_j.
}
\]
The quadruple $(\cU_i,\tU_i,\Gamma_i,\varphi_i)$ is called an \textbf{orbifold chart} and the collection of orbifold charts is called an \textbf{orbifold atlas}.  
\end{enumerate} 
\label{sfs.1}\end{definition}     
We say an orbifold is \textbf{effective} if, for all orbifold charts
$(\cU_i,\tU_i,\Gamma_i,\varphi_i)$, the action of $\Gamma_i$ on
$\tU_i$ is effective, i.e.\ the identity of $\Gamma_{i}$ is the only
element which acts as the identity map.   On an orbifold, we can consider bundles which are compatible with the orbifold structure.

\begin{definition}
Let $B$ be an orbifold and $\tF$ a smooth manifold.  Then a smooth
\textbf{orbibundle} with fibre $\tF$ is an orbifold $Y$ together with
a surjective continuous map $\pi\colon Y\to B$ satisfying the following:
\begin{enumerate}
\item[(i)]  $B$ admits an orbifold atlas of orbifold charts $(\cU_i,\tU_i,\Gamma_i,\varphi_i)$ such that for each $i$,  there is a smooth action of $\Gamma_i$ on $\tU_i\times \tF$ making the projection 
$\tU_i\times \tF\to \tU_i$ $\Gamma_{i}$-equivariant, and there is a homeomorphism 
$\psi_i\colon (\tU_i\times \tF)/\Gamma_i\to \pi^{-1}(\cU_i)$ inducing the following commutative diagram,
\[
\xymatrix{
  \tU_i\times \tF \ar[d]\ar[r] & (\tU_i\times \tF)/\Gamma_i \ar[r]^{\psi_i}\ar[d] & \pi^{-1}(\cU_i) \ar[d] \\
  \tU_i \ar[r] & \tU_i/\Gamma_i \ar[r]^{\varphi_i} & \cU_{i}. 
}
\]
\item[(ii)]  Whenever $\cU_i\subset \cU_j$, there exists a $\Gamma_i$-equivariant embedding 
$\widetilde{\psi}_{ij}\colon \tU_i\times \tF\to \tU_j\times \tF$  making the following diagram commutes,
\[
\xymatrix{
\tU_i \ar[d]^{\widetilde{\varphi}_{ij}} & \ar[l] \tU_i\times \tF \ar[d]^{\widetilde{\psi}_{ij}} \ar[r] & (\tU_i\times \tF)/\Gamma_i \ar[rr]^{\psi_i} \ar[d]^{\psi_{ij}=\widetilde{\psi}_{ij}/\Gamma_i} & & \pi^{-1}(\cU_i) \ar@{^{(}->}[d] \\
\tU_j & \ar[l] \tU_j\times \tF \ar[r] &(\tU_j\times \tF)/\Gamma_i\ar[r] & (\tU_j\times \tF)/\Gamma_j \ar[r]^{\psi_j} & \pi^{-1}(\cU_j).
}
\]
\end{enumerate}
In particular, the collection $(\pi^{-1}(\cU_i), \tU_i\times \tF, \Gamma_i, \psi_i)$ form an orbifold atlas for $Y$.  
\label{sfs.2}\end{definition}
\begin{remark} 
 Without loss of generality, we can assume that in the orbifold chart 
 $(\pi^{-1}(\cU_i), \tU_i\times \tF, \Gamma_i, \psi_i)$,  the action of $\Gamma_i$ on $\tU_i\times \tF$ is induced by the action on $\tU_i$ and a fixed action on $\tF$,
 \[
      (u,f)\cdot \gamma= (u\cdot \gamma, f\cdot \gamma), \quad u\in\tU_i, \ f\in \tF, \ \gamma\in \Gamma_i.
\]
Indeed, if the action is not already of this form, one can use a $\Gamma_i$-equivariant connection for the fibration $\tU_i\times \tF\to \tU_i$ and parallel transport to identify each fibre above $\tU_i$ with a fixed one.  With such an identification, the
action of $\Gamma_i$ is of the desired form. 
\label{sfs.5}\end{remark}
We can now define Seifert fibrations in terms of orbibundles.
\begin{definition}
A \textbf{Seifert fibration} is an orbibundle for which the total
space $Y$ is a smooth manifold, that is, such that the $\Gamma_i$ actions on $\tU_i\times \tF$ are free for each $i$.  
\label{sfs.4}\end{definition}

\begin{remark}
As explained in \cite[Remark~1.3]{Goette2011}, there is no loss of
generality in assuming that the
space of leaves of a Seifert fibration is an effective orbifold.  For
this reason, we will assume that all the orbifolds considered are effective unless otherwise specified.  
\label{sfs.4b}\end{remark}

In Section \ref{wtl.0}, we will consider the following special case. 

\begin{definition}\label{thm:goodassumption}
A foliation $\mathcal{F}$ on $ Y$  is a \textbf{good Seifert fibration} if it arises as follows:
\begin{itemize}
\item[(i)] $Y= \widetilde{Y}/\Gamma$ where $\widetilde{Y}$
  is a smooth compact manifold on which a finite group $\Gamma$ acts on the right by diffeomorphisms freely and
properly discontinuously.  The manifold $\widetilde{Y}$ is the total space of a 
fibration 
\begin{equation}
\xymatrix{
                 \widetilde{F} \ar[r] & \widetilde{Y} \ar[d]^{\wt{\pi}} \\
                                              & \wt{B}
                                              }
\label{rf.1}\end{equation}
where the base $\wt{B}$ is a closed manifold and the fibre $\widetilde{F}$ is 
a smooth manifold.  The group $\Gamma$ acts smoothly on
$\wt{B}$ in such a way that 
\[
   \wt{\pi}(y\cdot\gamma) =  \wt{\pi}(y)\cdot\gamma,  \quad \forall \ y\in \widetilde{Y}, \ \gamma\in \Gamma.  
\]
\item[(ii)] The leaves of the foliation $\mathcal{F}$ are given by the images of the fibres of the fibration $\wt{\pi}\colon \widetilde{Y}\to \wt{B}$ under
the quotient map
\begin{equation}\label{eq:pMquotient}
q_{\widetilde{\pi}}\colon \widetilde{Y}\to Y.
\end{equation}
Thus the leaves of the foliation
are given by $q_{\widetilde{\pi}}( \wt{\pi}^{-1}(y))$ for $y\in \wt{B}$.
\end{itemize} 
\label{tfb.2b}\end{definition}

Let $\overline{M}$ be a smooth manifold with boundary $\pa M$ equipped with a Seifert fibration $\cF$.  Denote by $M$ the interior of $\overline{M}$.  Let
$B$ be the space of leaves and let 
$\pi\colon \pa M\to B$ be the associated projection.    Since $B$ is an
orbifold, we know (see for instance \cite[\S~4.4.10]{Pflaum2001}) that it has a natural induced
stratification  defined in terms of the isotropy types of the points on
$B$, which are defined as follows.  If
$(\cU, \tU,\Gamma,\varphi)$ is an orbifold chart for an open
neighborhood $\cU$ of a point $p\in B$, then the isotropy type of $p$,
denoted $[I_p]$, is the isomorphism class
of the isotropy group $I_{\widetilde{p}}$ of $\widetilde{p}$
under the action of $\Gamma$, where $\widetilde{p}$ is any
point in $\varphi^{-1}(p)$.  As can be readily checked, the
isotropy type depends neither on the choice of $\widetilde{p}$ nor on the
orbifold chart.  In terms of the isotropy type,
the stratification of $B$ is given by
\begin{equation}
      B=B_b\supset B_{b-1}\supset \cdots \supset B_1 \supset B_0, \quad \mbox{with}\, b= \dim B, 
\label{strat.1}\end{equation} 
where the stratum $B_i$ is the disjoint union of all the strata of
dimension $i$ of the form 
\[
   s_{[\Gamma]} = \overline{  \{p \in B  : [I_p]= [\Gamma]   \}    },
\] 
where the closure is taken in $B$ and $[I_p]=[\Gamma]$ means that the group
$I_p$ is isomorphic to the group  $\Gamma$.  That \eqref{strat.1} is a well-defined stratification can be
seen locally in an orbifold chart by using the corresponding result
about the stratification of manifolds admitting a smooth group action,
see for instance \cite{DK2000}.  Recall that the depth
$\delta(s_{[\Gamma]})$ of a stratum $s_{[\Gamma]}$ of $B$ is the biggest
integer $k$ such that there exists different strata
$s_0,\ldots,s_{k-1}$ such that 
\[
   s_0 \subsetneqq s_1 \subsetneqq \cdots \subsetneqq s_{k-1} \subsetneqq s_k := s_{[\Gamma]}.
\]
The depth of $B$ is 
\[
  \delta(B)= \sup \{ \delta(s_{[\Gamma]}) : s_{[\Gamma]}\; \mbox{is a stratum of}\ B    \}.
\]
Let $X$ be the singular space given in \eqref{eq:X}.  Then $X$ is
naturally a stratified space with stratification 
\[
   X_n\supset X_{n-1}\supset X_{n-2}\supset \cdots \supset X_1 \supset X_0
\] 
induced by of $B$, i.e.\
\[
   X_j := \left\{ \begin{array}{ll}  X, & j=n,  \\
                                                      B, &  b\le j < n, \\
                                                      B_j, & j< b.                            
                        \end{array}
    \right.
\]
With this stratification, the depth of $X$ is one more than that of
$B$, i.e.\ $\delta(X)=\delta(B)+1$.  

We will now discuss regular neighborhoods for points $p \in B \subset
X$.  We will use the following standard notation for cones; given
$\epsilon > 0$ and a topological space $Z$, set $$C_{\epsilon}(Z) := Z
\times [0, \epsilon] / Z \times \set{0}.  $$

\begin{lemma}
Given $p\in B\subset X$, let $\Gamma$ represent the isotropy type
$[I_p]$ of $p$ in $B$.  Let $b - d$ be the smallest dimension of a
stratum
$s_{[\Gamma]}$ in $B$ containing $p$.  Then there exists a smooth
compact $\Gamma$-manifold $\tF$ and a smooth $\Gamma$ action on the
unit ball  $\bbB^{d}$ such that the $\Gamma$ action on
$\bbB^{d}\times \tF$ is free and $p$ has a regular neighborhood in $X$  of the
form
\begin{equation}\label{eq:regneighb}
     U= \bbB^{b-d} \times C_{1}(L),
\end{equation}
where $L= \tL/\Gamma$ and $\tL$ is the stratified space, with the obvious induced $\Gamma$ action, obtained from $\bbB^{d}\times \tF$
by collapsing the boundary fibration $\pa\bbB^{d}\times \tF\to \pa
\bbB^{d}$ onto its base.  If $d = 0$, then $\wt{L} = \wt{F}$.  See
Figure \ref{fig:widetildeL}.  Furthermore, 
\begin{equation}
  \label{eq:deeper}
  \delta(L) < \delta(X).
\end{equation}
\label{regn.1}\end{lemma}
\begin{proof}
Let $V \subset
B$ be an open neighborhood of $p$ and $(V, \tV,\Gamma, \varphi)$ an associated orbifold chart.  Without loss of generality, we can assume that $V$ and the associated orbifold chart are chosen in such a way that there is a unique $\wt{p} \in \wt{V}$ such that $\varphi\circ q(\wt{p}) = p$, where 
  $q\colon \tV\to \tV/\Gamma$ is the quotient map.  This means the action of $\Gamma$ on $\tV$ fixes $\wt{p}$, so that the isotropy type of $p$ is precisely the isomorphism class of the group $\Gamma$. 
By taking $\tV$ even smaller if needed, we know by the tube theorem
for smooth Lie group actions (see for instance Theorem~2.4.1 in
\cite{DK2000}), that we can assume $\tV$ is a linear representation of
$\Gamma$.  Putting an invariant inner product on $\tV$, we thus see there is a
decomposition $$\tV= \tV^{\Gamma}\times \tV'$$ obtained by taking the
orthogonal projection on the subvector space
$\tV^{\Gamma}$ of points fixed by $\Gamma$.  Note that if $d = 0$ then
$\tV = \tV'$.
By Remark~\ref{sfs.5} and at the cost of taking $V$ even smaller, we
can assume the Seifert fibration on $\pa M$ restricted to $V$ lifts on
$\tV$ to a $\Gamma$-equivariant fibration,
$$
    \pr\colon \tV\times \tF \to \tV, 
$$    
where $\tF$ is a smooth compact $\Gamma$-manifold and $\pr$ is the
projection on the first factor.    The smooth action of
$\Gamma$ on the total space $\tV\times \tF$ is free and we thus have a
quotient map $q\colon \tV\times \tF\to W$ onto the total space $W\subset
\pa M$ of the Seifert fibration lying above $V\subset B$.

Let $\bbB^{\tV^{\Gamma}}$ be the open unit ball in $\tV^{\Gamma}$.
Let $\bbB^{\tV'\times \bbR_x}$ be the unit ball in $\tV'\times \bbR_x$
and let $\bbB^{\tV'\times \bbR_x}_+ =\bbB^{\tV'\times \bbR_x}\cap \{x\ge
0\}$ be the corresponding half-ball.  With $\Gamma$ acting trivially on
$[0,\epsilon)_x$ and assuming without loss of generality that
$\epsilon>1$, we see that 
$\bbB^{\tV^{\Gamma}}\times \bbB^{\tV'\times \bbR_x}_+$ is an open subset
of $\widetilde{V}\times [0,\epsilon)_x$ which is preserved by the action of
$\Gamma$.  Thus, it descends to an open subset in $W\times
[0,\epsilon)_x\subset \overline{M}$.  Under the collapsing map $c_{\pi}\colon \overline{M}\to
X$, this open subset  is mapped to a regular neighborhood $U$ of $p\in X$.  To see
this,  notice that under the identification
$\bbB^{\tV'\times \bbR_x}_+\cong C_1(\bbS^{\tV'\times \bbR_x}_+)$,  where
$\bbS^{\tV'\times \bbR_x}_+$ is the corresponding half-sphere,  the
action of $\Gamma$ on $C_1(\bbS^{\tV'\times \bbR_x}_+)$ is induced from
an action on $\bbS^{\tV'\times \bbR_x}_+$.  Similarly, the action of
$\Gamma$ on  $C_1(\bbS^{\tV'\times \bbR_x}_+ \times \tF)$ is induced from
a corresponding action on
$\bbS^{\tV'\times \bbR_x}_+\times \tF$.  This action descends to the
stratified space $\tL$ given by
\begin{equation}\label{eq:wtL}
  \tL=  \bbS^{\tV'\times \bbR_x}_+\times \tF/\sim,
\end{equation} 
where
\begin{equation*}
  (s_1,f_1)\sim (s_2,f_2) \ \Longleftrightarrow \  (s_1,f_1) = 
  (s_2,f_2) \; \mbox{or} \; s_1=s_2 \in \pa  \bbS^{\tV'\times \bbR_x}_+,
\end{equation*}
see Figure \ref{fig:widetildeL}.  If $L= \tL/\Gamma$ is the
corresponding stratified space obtained by taking the quotient with
respect to this action,  then the above identification shows there is
a natural homeomorphism 
$$
      U \cong  \widetilde{U}^{\Gamma}\times C_1(L).
$$   
If $d = 0$, then $\tV' = \set{0}$ and $\wt{L} = \wt{F}$.
Finally, to establish \eqref{eq:deeper}, note that if $d > 0$, then the depth of $L$ is given by
$$
   \delta(L)= 1+ \delta(\pa\bbS^{\tV'\times \bbR_x}_+/\Gamma) 
$$
and that the depth of $\pa\bbS^{\tV'\times \bbR_x}_+/\Gamma$ is strictly
less than the depth of $B$.  Thus
\eqref{eq:deeper} follows from $\delta(X) = \delta(B) + 1$ in this
case.  If $d = 0$, then
$L$ is smooth, so $\delta(L)=0<1=\delta(X)$ and \eqref{eq:deeper} follows again.  
\end{proof}

\begin{figure}
  \centering
  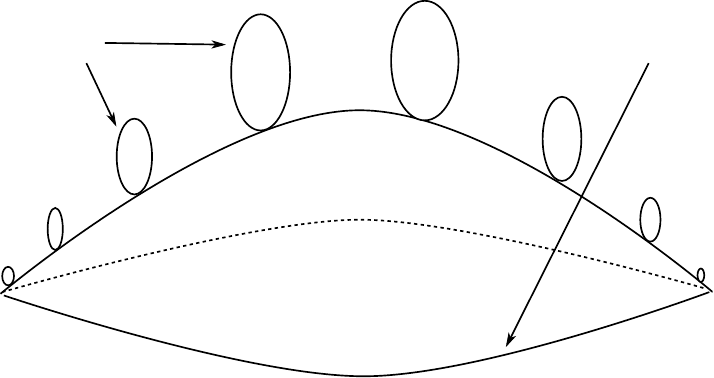
  \caption{the space $\wt{L}$ from \eqref{eq:wtL}, where
    $\bbS^{\tV'\times \bbR_x}_+$ is written simply as $\mathbb{B}^{d}$}
  \label{fig:widetildeL}
\end{figure}

\begin{definition}\label{thm:resolution}
  We will refer to the space $\widetilde{U} = \bbB^{b-d}
\times C_{1}(\wt{L})$ with its natural action of $\Gamma$ and quotient
map $q \colon \widetilde{U} \lra U$ as a \textbf{resolution} of $U$.  
\end{definition}

\section{From intersection cohomology to weighted cohomology}\label{sec:weighttoint}

We will now relate the intersection
cohomology of $X$ to the cohomology of a complex of sheaves of
weighted $L^{2}$ differential forms.  The proof will use sheaf theory and will be local on the
space of leaves of the boundary foliation.  For this reason,  we do not yet require that the Seifert fibration on the boundary be good.  The main ingredient reducing the discussion to local considerations is the following fundamental result from \cite[p.104]{GM1983}, \cf Proposition~1 in \cite{HHM2004}.  

\begin{proposition}[\cite{GM1983}]\label{thm:localtoglobal}
  Let $X$ be a stratified space, and let $(\mathcal{L}^{*}, d)$ be a
  complex of fine sheaves on $X$ with cohomology $H^{*}(X,
  \mathcal{L})$.  Suppose that if $U$ is a neighborhood in the principal
  (smooth) stratum of $X,$ then $H^{*}(U,
  \mathcal{L}) = H^{*}(U, \mathbb{C})$, while if $q$ lies in a stratum
  of codimension $\ell$ and $U = V \times C_{1}(L)$ as in \eqref{eq:regneighb}, then
  \begin{equation}
    \label{eq:localequiv}
    H^{*}(U, \mathcal{L}) \simeq I\!H^{k}_{\mathfrak{p}}(U) =
    \left\{
      \begin{array}{cc}
        I\!H^{k}_{\mathfrak{p}}(L), & k \leq \ell - 2 - \pp(\ell), \\
        0, & k \geq \ell - 1 - \pp(\ell).
      \end{array}
      \right.
  \end{equation}
Then there is a natural isomorphism between the hypercohomology
$\mathbb{H}^*{(X, \mathcal{L}^{*}})$ associated to this complex of sheaves
and $I\!H^*_{\mathfrak{p}}(X)$, the intersection cohomology with
perversity $\mathfrak{p}$.
\end{proposition}

Before we continue, we take a moment to give the general definition of
foliated boundary and foliated cusp metrics, following \cite{pomfb}.
Foliated boundary
metrics are metrics on a vector bundle over the compact manifold with
boundary $\overline{M}$, defined as follows.  First, for a fixed boundary
defining function $x$ and a choice of foliation $\cF$ on $\pa\overline{M}$, we have the space of foliated cusp vector
fields
\begin{equation*}
  \label{eq:1}
  \mathcal{V}_{\mathcal{F}}(M) = \set{ \xi \in \Gamma(\mbox{T}\overline{M}) : \xi
    x = x^{2}C^{\infty}(\overline{M}) \mbox{ and } \xi \rvert_{\p M}
    \in \Gamma(\mbox{T}\mathcal{F}) }.
\end{equation*}
This set of vector fields over $\overline{M}$ is a Lie algebra under Lie bracket.  
It can be identified with the space of smooth sections of the vector bundle
${}^{\cF}T\overline{M}$,  whose fibre at each point $p \in \overline{M}$ is the set
$\mathcal{V}_{\mathcal{F}}(M) / \lp \mathcal{I}_{p}(\overline{M})
\mathcal{V}_{\mathcal{F}}(M)\rp$, where $\mathcal{I}_{p}(\overline{M})
\subset C^{\infty}(\overline{M})$ is the ideal of smooth functions
vanishing at $p$.  Then a \textbf{foliated boundary metric} is a
Riemannian metric $g_{\cF}$ on $M$ induced by the identification
$TM=\left.{}^{\cF}T\overline{M}\right|_{M}$ and a choice of a
Euclidean metric on the bundle ${}^{\cF}T\overline{M}$.  A
\textbf{foliated cusp metric} is a
Riemannian metric on $M$ of the form $g_{\mathcal{F}-c} = x^{2}
g_{\mathcal{F}}$ for some foliated boundary metric $g_{\mathcal{F}}$ on $M$.  When $\cF$ is a Seifert fibration, examples of such metrics are given by the exact foliated boundary and exact foliated cusp metrics described in the introduction.

Given an open set $U \subset X$, Let $L^{2}\Omega^{k}(g_{\foc}, U)$ denote the completion of
the space of compactly supported differential $k$-forms on $U - B$
with respect to the $L^{2}$-norm induced by $g_{\foc}$, and let
$x^{a}L^{2}\Omega^{k}(g_{\foc}, U)$ denote the space of currents
$\omega$ satisfying $x^{-a}\omega \in L^{2}\Omega^{k}(g_{\foc}, U)
$.   By considering the subspace 
$$
x^{a}L^{2}_d\Omega^{k}(g_{\foc}, U) := \{ \omega\in  x^{a}L^{2}\Omega^{k}(g_{\foc}, U) \: \
    d\omega\in x^{a}L^{2}\Omega^{k+1}(g_{\foc}, U)\},$$
we obtain a complex
\begin{equation}
  \label{eq:localweightedcomplex}
\cdots \xrightarrow{d} x^{a}L_d^{2}\Omega^{k}(g_{\foc}, U) \xrightarrow{d}
  x^{a}L_d^{2}\Omega^{k + 1}(g_{\foc}, U) \xrightarrow{d} \cdots.
\end{equation}
The corresponding cohomology groups are given by
\begin{equation}
  \label{eq:localweightedhomology}
  W\!H^{k}(g_{\foc}, a, U) = \frac{ \set{\omega \in
    x^{a}L^{2}\Omega^{k}(g_{\foc}, U) : d \omega = 0}}
    {\set{ d \zeta \in x^{a}L^{2}\Omega^{k}(g_{\foc}, U) : \zeta
      \in x^{a}L^{2}\Omega^{k - 1}(g_{\foc}, U)}}.
\end{equation}
For a foliated boundary metric, $g_{\fob}$, we consider instead the complex
\begin{equation}
  \label{eq:localweightedcomplexfob}
\cdots \xrightarrow{d} x^{a-1}L_d^{2}\Omega^{k-1}(g_{\fob}, U)\xrightarrow{d} x^{a}L^{2}_d\Omega^{k}(g_{\fob}, U) \xrightarrow{d}
  x^{a+1}L_d^{2}\Omega^{k + 1}(g_{\fob}, U) \xrightarrow{d} \cdots,
\end{equation}
where this time
$$
x^{a}L^{2}_d\Omega^{k}(g_{\fob}, U)= \{
\omega\in x^{a}L^{2}\Omega^{k}(g_{\fob}, U)\ : \  d\omega\in x^{a+1}L^{2}\Omega^{k + 1}(g_{\fob}, U)   \}.
$$
The corresponding cohomology groups are then denoted by
$$
\mathcal{W}\!H^{k}(g_{\fob}, a, U):=   \frac{ \set{\omega \in
    x^{a}L^{2}\Omega^{k}(g_{\fob}, U) : d \omega = 0}}
    {\set{ d \zeta \in x^{a}L^{2}\Omega^{k}(g_{\fob}, U) : \zeta
      \in x^{a-1}L^{2}\Omega^{k - 1}(g_{\fob}, U)}}.
$$
For a fixed boundary foliation $\cF$ and a fixed boundary defining function
$x\in \mathcal{C}^{\infty}(\overline{M})$, the space
$x^{a}L_d^{2}\Omega^{k}(g_{\fob}, U)$ and thus  the group
$\mathcal{W}\!H^{k}(g_{\fob}, a, U)$  do not depend on the choice of
the foliated boundary metric $g_{\fob}$.  Indeed, if $g_{\fob}$ and
$g_{\fob}'$ are two foliated boundary metrics,
then by the compactness of $\overline{M}$ there exists a constant $C>0$ such that 
\[
      \frac{g_{\fob}}{C} \le  g_{\fob}' \le C g_{\fob},
\]
so the weighted $L^{2}$ spaces, and thus the cohomologies, are the
same for either metric.
The same invariance holds for foliated cusp
metrics.  In particular, we may assume that $x^{2}g_{\fob} =
g_{\foc}$. It follows that
$x^{a}L^{2}\Omega^{k}(g_{\fob}, U)  = x^{n/2 - k + a}L^{2}\Omega^{k}(g_{\foc},
U) $, and consequently that
\begin{equation}
  \label{eq:fobequivfoc}
   \mathcal{W}\!H^{k}(g_{\fob}, a, U)  =  W\!H^{k}(g_{\foc}, (n/2) - k + a, U) .
\end{equation}
Because we will work predominantly with fibred cusp metrics
in this section, and because the particular choice of 
fibred cusp metric does not effect the weighted cohomology, we will often write simply
\begin{align*}
  W\!H^{k}(a, U) &:= W\!H^{k}(g_{\foc}, a, U).
\end{align*}
As explained in
Section~2.1 of \cite{HHM2004} using the Kodaira decomposition theorem, the same invariance under quasi-isometries holds for Hodge
cohomology.   For instance, for any two  foliated cusp metrics $g_{\foc}$ and $g_{\foc}'$ associated to the same choices of foliation $\cF$ and boundary defining function $x$,
$L^2\mathcal{H}^k(M, g_{\foc}') \cong L^2\mathcal{H}^k(M, g_{\foc})$.

The purpose of this section is to prove the following analogue of
Proposition 2 in \cite{HHM2004}.
\begin{theorem} \label{thm:wcihfol} 
Let $M$ be a
  manifold with boundary foliated by a Seifert fibration $\cF$,
  and let $g_{\foc}$ be an associated foliated cusp metric on $M$.  Suppose that
  $k-1+a-f/2 \neq 0$ whenever $H^k(\tF)\ne 0$.
Then 
\begin{equation}\label{eq:localequivweighted}
W\!H^*(g_{\foc}, a, M) \simeq I\!H^*_{[a+ f/2]}(X, B),
\end{equation}
where $[a+f/2]$ is the greatest integer less than or equal to
$a+f/2$.
\end{theorem}
We notice two easy consequences of the theorem.
\begin{corollary}
  If $k - 1 + a - f/2 \neq 0$ whenever $H^k(\tF)\ne0$, then the
  differential in the complex \eqref{eq:localweightedcomplex} has
  closed range.  
\label{dclosed.1}\end{corollary}
\begin{proof}
By the theorem, $W\!H^*(g_{\foc}, a, M)$ is finite dimensional, so the differential must have closed range.
 \end{proof}

\begin{corollary}
The intersection cohomology $I\!H^*_{\pp}(X)$ depends only on $\pp(f+1)$. 
\label{icp.1}\end{corollary}
\begin{proof}
Use the identification \eqref{eq:localequivweighted} with $a\in \bbR$ such that $2a \notin \mathbb{Z}$ and $[a+ \frac{f}{2}]= \pp(f+1)$.
\end{proof}

The main tool in the proof of Theorem \ref{thm:wcihfol} is Proposition
\ref{thm:localtoglobal}.  To apply this proposition, we must know that the sheaf which associates to each open set $U
\subset X$ the complex in \eqref{eq:localweightedcomplex}
is fine.  Indeed, given an open cover $\set{ O_{k} }_{k \in K}$, we can choose a
subcover $\set{ U_{i} }_{i \in I}$ so that each $U_{i}$ is a regular
neighborhood as in \eqref{eq:regneighb} with a resolution $q\colon \wt{U}_{i}
\lra U_{i}$. Let
$\set{V_{i}}_{i \in I}$ be yet another subcover so that $\overline{V_{i}} \subset
U_{i}$.  Let $\tV_i= q^{-1}(V_i)$.  If $U_i$ is supported away from
$B$, let $\rho_i\in \mathcal{C}^{\infty}(U_i)$ be a nonnegative function with compact support
equal to $1$ on    $V_i$.  On the other hand, if $U_i$ intersects $B$, we are assuming that 
$\widetilde{U}_i$ is of the form $\wt{U}_i=\mathbb{B}^{b} \times C_{\epsilon}(\wt{F})$.  Thus, there is 
a projection $\widetilde{\pi}: \wt{U}_i\to \mathbb{B}^{b}\times [0,\epsilon)_x$ given by
$$
   \widetilde{\pi}(b,x,f)= (b,x).
$$ 
Let $\phi_i\in \mathcal{C}^{\infty}(\mathbb{B}^{b}\times
[0,\epsilon)_x)$ be a nonnegative smooth function
with compact support such that $\widetilde{\pi}^*\phi_i$ is equal to $1$ on $\tV_i$ and $\phi_i$ is constant near $x=0$.  
Clearly, since $\widetilde{\pi}^*\phi_i$ is constant along $\tF$, the
norm of $d(
\widetilde{\pi}^*\phi_i)$ is bounded with respect to a choice of 
fibred cusp metric on $\widetilde{U}_i$.  Averaging with respect to $\Gamma$, we can assume 
$\phi_i$ is $\Gamma$-invariant.  Thus, the function  $ \widetilde{\pi}^*\phi_i$ descends to
a smooth function $\rho_i$ with compact support on $U_i$.  As usual,
the functions $\chi_{i} := \rho_{i} / \sum \rho_{j}$ form a partition
of unity subordinate to the cover  $\set{ O_{k} }_{k \in K}$.  The
fact that the differential $d\chi_i$ has norm bounded with respect to
a choice of foliated cusp metric insures that it acts by multiplication on $x^{a}L^{2}_d\Omega^{k}(g_{\foc}, U_i)$.  Thus, the associated sheaf is fine.

As a preliminary to the proof of Theorem \ref{thm:wcihfol},
we discuss the computation of the weighted cohomology of
$C_{\epsilon}(\wt{F})$ from Proposition 2 of \cite{HHM2004}. There, they show that, provided
\begin{equation}\label{eq:acondition}
k - 1 + a - f/2 \neq 0 \quad \mbox{whenever} \quad H^k(\tF)\ne 0,
\end{equation}
we have
\begin{equation}\label{eq:coneFcoho}
  W\!H^{k}(a, C_{\epsilon}(\wt{F})) \simeq 
    \left\{
      \begin{array}{cl}
       H^{k}(\wt{F}) & \mbox{ if } k < f/2 - a \\
        0 & \mbox{ if } k \geq f/2 - a .
      \end{array}
      \right. 
\end{equation}
If $\iota \colon \wt{F} \lra
C_{\epsilon}(\wt{F})$ is the inclusion of $\wt{F}$ into $\wt{F} \times
\set{\epsilon}$, then the isomorphism with $H^{k}(\wt{F})$ in the case
$k < f/2 - a$ is induced by $\iota$.  Since we will need the proof in a moment, we recall it now but direct
the reader to \cite{HHM2004} for more
details.  By the $L^{2}$-K\"unneth theorem from \cite{Z1982}, as long
as \eqref{eq:acondition} holds, $W\!H^k(a, C_{\epsilon}(\wt{F})) $ is
isomorphic to
\begin{equation}\label{eq:Kunneth}
\begin{array}{rl}
  &\lp W\!H^0\lp (0,\epsilon),dx^2/x^2,k+a-f/2 \rp
  \otimes H^k(\wt{F}) \rp \oplus \\
 & \
 \lp W\!H^1 \lp (0,\epsilon),dx^2/x^2,k-1+a-f/2 \rp  \otimes H^{k-1}(\wt{F}) \rp.
\end{array} 
\end{equation}
It is trivial to check that $ W\!H^1 \lp (0,\epsilon),dx^2/x^2,  b \rp = 0$
if $b \neq 0$, and that $ W\!H^0\lp (0,\epsilon),dx^2/x^2, b\rp$ is $0$ for
$b \geq 0$ and $\mathbb{C}$ for $b < 0$.  Combining these two facts
with \eqref{eq:Kunneth} proves \eqref{eq:coneFcoho}, and the
naturality of the K\"unneth formula shows that the isomorphism is induced by
the inclusion $\iota$.

For the space $\wt{L}$ in Lemma \ref{regn.1}, if $d= 1$ we can use the
exact same argument to compute its weighted cohomology.  In this case, $\wt{L}
= [0,1]_{x} \times \wt{F} / (\set{0} \times \wt{F} \cup \set{1} \times
\wt{F})$, and by 
invariance, we can take the fibred cusp metric on $\wt{L}$ to be
\begin{equation}\label{eq:metricontildeL}
g = f^{-1}(x) dx^{2} + f(x)\wt{k},
\end{equation}
where $f$ is a smooth positive
function on $(0,1)$ with
\begin{equation}
  \label{eq:3}
  f(x) = \left\{
    \begin{array}{ccc}
      x^{2} &\mbox{for}& x < 1/4 \\
      (x - 1)^{2} &\mbox{for}& x > 3/4, \\
    \end{array} \right.
\end{equation}
and $\wt{k}$ is a metric on $\wt{F}$.  An argument identical to
that of the previous paragraph yields 
\eqref{eq:coneFcoho} with $C_{\epsilon}(\wt{F})$ replaced by
$\wt{L}$ in \eqref{eq:coneFcoho}
and $dx^{2}/x^{2}$ replaced by $dx^{2}/f$ in \eqref{eq:Kunneth}.  In fact, if for small
$\epsilon > 0$, we consider $C_{\epsilon}(\wt{F})  \xrightarrow{\iota} \wt{L}$, where 
$\I(x,f) =  (x,f)$, then we have shown that
\begin{equation}
  \label{eq:d=1}
      \iota^{*} \colon  W\!H^{k}(a, \wt{L}) \lra W\!H^{k}(a, C_{\epsilon}(\wt{F})) \\
\end{equation}
is an isomorphism (provided \eqref{eq:acondition} holds).

To obtain a similar result for $d>1$,  consider the following sequence of inclusions
\begin{equation}\label{eq:maps}
  \xymatrix{
    C_{\epsilon}(\wt{F}) \ar[r]^-{\I} & \wt{L} \ar[r]^-{j}& C_{1}(\wt{L})
    }
\end{equation}
where $\I$ is defined as follows.  Near the boundary of $\mathbb{B}^{d}$,
let $(x, y)$ be coordinates where $x$ is a b.d.f.\ and $y$ is a
coordinate on the sphere.  Then set 

\begin{equation}
  \label{eq:42}
  \begin{split}
    \I(x, f) &:= (x, y, f) \mbox{ for some fixed but arbitrary } y.
    \end{split}
\end{equation}
We endow $\wt{L}$ with a foliated cusp metric by considering the
foliated cusp metric $g_{\foc}$ of $X$ on a
canonical neighborhood $U \simeq \mathbb{B}^{b - d} \times C_{1}(L)$ as in
\eqref{eq:regneighb}, pulling $g_{\foc}$ to $\set{p_{0}}\times C_1(L)$
for a fixed $p_{0} \in \bbB^{b-d}$, i.e.\ to the inclusion of a cone
$C_1(L)$ over an arbitrary point in $\bbB^{b-d}$, and then lifting to
$C_1(\wt{L})$.  Finally, endow  $C_{\epsilon}(\wt{F})$ with the pullback
metric induced by $j\circ \I$ in \eqref{eq:maps}.

\begin{lemma} If $k-1+a-\frac{f}{2}\neq 0$ whenever $H^k(\tF)\ne 0$, then the induced
mappings
\begin{equation}
  \label{eq:4}
  \begin{split}
    \I^{*} \colon  W\!H^{k}(a, \wt{L}) &\lra W\!H^{k}(a, C_{\epsilon}(\wt{F})) \\
    j^{*} \colon  W\!H^{k}(a, C_{1}(\wt{L})) &\lra  W\!H^{k}(a, \wt{L}).
  \end{split}
\end{equation}
are isomorphisms.  
\label{abp.1}\end{lemma}
\begin{proof}
To see that $\I^*$ is an isomorphism, we will proceed by induction on $d$,
\eqref{eq:d=1} being the base case.  Assume the result holds for
$d\le m-1$.  We need to show the result holds for $d=m$.  Let
$(y^1,\ldots,y^d)$ denote the standard coordinates on  $\bbR^d$ and
consider the open sets in $\bbB^d$ given by
\begin{equation}
\begin{gathered}
 U= \set{ y= (y^1,\ldots, y^d)\in \bbB^d : y^1 < 1/2 } \\
 V= \{ y= (y^1,\ldots, y^d)\in \bbB^d : y^1 > - 1/2\}.
 \end{gathered}
\end{equation}
Let $\widetilde{U}= \pr^{-1}(U)$ and $\tV= \pr^{-1}(V)$ be the corresponding open sets in $\tL$, where 
$\pr\colon \tL= \bbB^d\times \tF/\sim \, \to \bbB^d$ is the natural projection.
Since $\tL= \widetilde{U}\cup \tV$, there is an induced Mayer-Vietoris
sequence
\begin{equation}\label{eq:MayerV}
  \begin{split}
    &\cdots \xrightarrow{\wt{\Delta}} W\!H^{k-1}(a,\wt{U} \cap \wt{V})
    \xrightarrow{\wt{\delta}} W\!H^{k}(a,\wt{L})
    \xrightarrow{\wt{\psi}}   \\
    &\qquad W\!H^{k}(a,\wt{U} ) \oplus W\!H^{k}(a,\wt{V})
    \xrightarrow{\wt{\Delta}} W\!H^{k}(a,\wt{U} \cap \wt{V})
    \xrightarrow{\wt{\delta}} \cdots
  \end{split}
\end{equation}
Since $\bbB^{d-1}\times \tF/\sim$ is a deformation retract of
$\widetilde{U}\cap \tV$ through maps which preserve the stratification, we see by our induction
hypothesis that 
$$
    W\!H^{k}(a,\wt{U} \cap \wt{V}) \simeq W\!H^{k}(a, C_{\epsilon}(\wt{F})).
$$
Since $\widetilde{U}$ and $\tV$ both retract onto
$C_{\epsilon}(\wt{F})$ (again, preserving the stratification), we also have canonical identifications
\[
   W\!H^{k}(a,\wt{U})=W\!H^{k}(a, C_{\epsilon}(\wt{F})), \quad W\!H^{k}(a, \wt{V})=W\!H^{k}(a, C_{\epsilon}(\wt{F})).
   \]
In terms of these identifications, the map $\wt{\Delta}$ is given by
$$
   \begin{array}{ccl}
    \wt{\Delta} \colon  W\!H^{k}(a, C_{\epsilon}(\wt{F}))\oplus W\!H^{k}(a, C_{\epsilon}(\wt{F})) &\to &
       W\!H^{k}(a, C_{\epsilon}(\wt{F}))  \\
        (\omega_1,\omega_2) & \mapsto & \omega_1-\omega_2  \end{array}
$$
and is surjective.  The boundary homomorphism $\wt{\delta}$ is therefore trivial and the map
$\wt{\psi}$ gives an isomorphism
$$
   \wt{\psi}\colon W\!H^{k}(a, \wt{L})  \to \ker \wt{\Delta}\cong W\!H^{k}(a, C_{\epsilon}(\wt{F})).
   $$
Since $\wt{\psi}$ corresponds to the map $\I^*$ under the identification $\ker \wt{\Delta}\cong W\!H^{k}(a, C_{\epsilon}(\wt{F}))$, we see $\I^*$ induces an isomorphism in weighted cohomology.

To see that $j^*$ induces an isomorphism, notice that since $C_{\epsilon}(\wt{F})$ is a deformation retract of $C_{1}(\tL)$, we easily see that 
$$
   \I^*\circ j^{*} \colon  W\!H^{k}(a, C_{1}(\wt{L})) \lra W\!H^{k}(a, C_{\epsilon}(\wt{F}))   $$
is an isomorphism, and thus that $j^*= (\I^*)^{-1}\circ (   \I^*\circ j^{*})$ is an isomorphism as well. 
\end{proof}

\begin{proof}[Proof of Theorem \ref{thm:wcihfol}]
We proceed by induction on the depth
of $X$, the base case being that of a smooth manifold, where the theorem is trivial.  Assume that Theorem
\ref{thm:wcihfol} is proven for foliated cusp metrics of depth not
greater than $\delta-1$ and that $\delta(X) = \delta$.

Let $q \in X_{n - j +
  1} - X_{n - j }$ and choose a regular
neighborhood as described in Lemma~\ref{regn.1}, $U \simeq V \times
C_{1}(L)$.  By Proposition~\ref{thm:localtoglobal}, we need to show that $W\!H^{k}(a,U) \simeq
W\!H^{k}(a,C_{1}(L))$ is
isomorphic to $I\!H^{k}_{[a + f/2]}(U) \simeq
I\!H^{k}_{[a + f/2]}(C_{1}(L))$.  In a moment, we will show that
\begin{equation}\label{eq:conetonocone}
I\!H^{k}_{[a + f/2]}(C_{1}(L)) \cong I\!H^{k}_{[a + f/2]}(L).
\end{equation}
Since the depth of $L$ is smaller than that of $X$, we have by our induction hypothesis that
\begin{equation}\label{eq:inductionequiv}
    I\!H^{k}_{[a + f/2]}(L)\cong W\!H^{k}(a,L).
\end{equation}

Thus, it suffices, thanks to \eqref{eq:conetonocone}, to show that the inclusion $i \colon L \to C_{1}(L)$ induces an isomorphism in weighted cohomology.  To see this,  consider the commutative diagram
  \begin{equation}
    \label{eq:square}
    \xymatrix{
      W\!H^{k}(a, C_{1}(\wt{L})) \ar[r]^{j^*} \ar@/^/[d] & W\!H^{k}(a,\wt{L}) \ar@/^/[d] \\
      W\!H^{k}(a,C_{1}(L)) \ar[r]^{i^*} \ar@/^/[u] & W\!H^{k}(a,L). \ar@/^/[u]
    }
  \end{equation}
  The upward and downward pointing maps are, respectively, the
  pullback via the projection and the averaging map, and they are
  injective and surjective, respectively.  In fact, the images of the
  upward pointing maps are exactly the $\Gamma$-invariant
  elements of the cohomology (i.e.\ those which can be represented by
  $\Gamma$-invariant forms.)  By Lemma~\ref{abp.1}, we know that $j^*$ is an
  isomorphism, and it obviously identifies the $\Gamma$-invariant elements in
  the two groups.  Since these are exactly the images of the upward
  pointing maps, commutativity shows that $i^*$ is an
  isomorphism as well and the theorem holds for $X$.

To see that \eqref{eq:conetonocone} holds, note that using
 \eqref{eq:inductionequiv}, \eqref{eq:square}, and Lemma~\ref{abp.1}, we have
 \begin{equation*}
   I\!H^{k}_{[a + f/2]}(L) \simeq W\!H^{k}(a,L) \hookrightarrow
   W\!H^{k}(a,\wt{L}) \simeq W\!H^{k}(a,C_{\epsilon}(\wt{F})).
 \end{equation*}
In degree $k \geq \ell - 1 - p(\ell)$, the latter is zero by \eqref{eq:coneFcoho} and 
\begin{equation*}
  p(\ell) \leq p(f + 1) + \ell - (f + 1) \implies \ell - 1 - p(\ell)
  \geq -a + f/2,
\end{equation*}
so \eqref{eq:conetonocone} follows from
 \eqref{eq:localequiv} in Proposition \ref{thm:localtoglobal}.

\end{proof}

 Using a similar approach by induction on the depth together with Proposition~\ref{thm:localtoglobal}, one can show the following result about  the intersection cohomology
 of orbifolds.  This is well-known, but since we were unable to find a reference, we include a proof for completeness.  
 
 \begin{theorem}
   The intersection cohomology  of a compact orbifold $B$ is
   independent of the  choice of standard perversity $\pp$ and is given by 
   singular cohomology (with real coefficients), $$
   I\!H_{\pp}^*(B) \cong H^k(B;\bbR).
 $$ 
 \label{ico.1}\end{theorem}  
\begin{proof}
Really, we will show intersection cohomology is given by the (orbifold) de Rham cohomology, which is identified with the singular cohomology (with real coefficients).    
We will proceed by induction on the depth of the orbifold.  For an orbifold of depth $0$, the result is trivial.  Thus, assume the result is true for compact orbifolds of depth less than $\delta$ and suppose $B$ is a compact orbifold of depth $\delta$.  Let $\Omega^{*}$ be the complex of sheaves of smooth orbifold forms on $B$.  For a point $b\in B$, we can find a neighborhood of the form
$U= \bbB^k \times C_1(L)$, where $L$ is an orbifold of depth less than
the one of $B$ obtained by taking the quotient of the sphere $\bbS^{\ell-1}$ by the smooth action of some finite group $\Gamma$, $L= \bbS^{\ell-1}/\Gamma$.

Since $U$ is contractible, we have that
\begin{equation}
  H^k(\Omega^*(U))= \left\{   \begin{array}{ll}\bbR; & k=0, \\
    0, & k>0.  \end{array} \right.
\label{orbc.1}\end{equation}
On the other hand, by our induction hypothesis, the result holds on $L=\bbS^{\ell-1}/\Gamma$, so that
\begin{equation}
  I\!H_{\pp}^k(L) =H^k(\Omega^*(L))= \left\{   \begin{array}{ll}\bbR; & k=0, \ell-1, \\
    0, & \mbox{otherwise}.  \end{array} \right.
\label{orbc.2}\end{equation}
Since 
$$
    \pp(\ell)\ge 0 \; \Longrightarrow \; \ell-1-\pp(\ell)\le \ell-1,
$$
we thus see that  we can rewrite \eqref{orbc.1} as
\begin{equation}
  H^k(\Omega^*(U))= \left\{   \begin{array}{ll} I\!H_{\pp}^k(L)  ; & k\le \ell-2-\pp(\ell), \\
    0, & k\ge \ell-1-\pp(\ell).  \end{array} \right.
\label{orbc.3}\end{equation}
Since $b\in B$ was arbitrary, we conclude by Proposition \ref{thm:localtoglobal} that 
$H^k(\Omega^*(B))= I\!H_{\pp}^k(B)$ for all $k$, completing the inductive step.    
\end{proof}

Notice in particular that this implies the following known fact about orbifolds, \cf \cite{Chen-Ruan2004}.

\begin{corollary}
Poincar\'e duality holds for the singular cohomology with real coefficients of a compact oriented orbifold. 
\end{corollary}

\section{From weighted cohomology to $L^{2}$ cohomology.}\label{wtl.0}

In this section, we will prove that $L^{2}$ harmonic forms for foliated
cusp (and foliated boundary) metrics is isomorphic to the images of
maps of appropriately weighted cohomology groups.  The main technical
tool in the proof is a parametrix for the Hodge-deRham operator on
differential forms, for which we rely on the thesis of Vaillant
\cite{Vaillant}.  Since the thesis is lengthy and technical, we will
use the description of the parametrix construction in \cite{HHM2004}
as a guide, and in doing to it will become clear why we need the
assumption that $\cF$ is a good Seifert fibration as in Definition
\ref{thm:goodassumption} for the purposes of this section.

Let $D_{\fob}= d + \delta_{\fob}$ be the Hodge-deRham operator
associated to an exact foliated boundary metric $g_{\fob}$ on $M$.  Let
also $D_{\foc}$ denote the Hodge-deRham operator associated to the conformally related exact foliated cusp metric $g_{\foc}= x^2 g_{\fob}$.  To construct parametrices for these operators, we will make use of their restrictions to the tubular neighborhood $\pa M\times [0,\epsilon)_x$.  Denote by $\widetilde{D}_{\fob}:= q^*D_{\fob}$ and 
$\widetilde{D}_{\foc}:= q^*D_{\foc}$ the pull-back of these restrictions under the quotient map $q\colon \pa\widetilde{M}\times [0,\epsilon)_x\to \pa M\times [0,\epsilon)_x$.  Clearly, $\widetilde{D}_{\fob}$ and $\widetilde{D}_{\foc}$ are the Hodge-deRham operators associated to the pull-back metrics $q^*g_{\fob}$ and $q^*g_{\foc}$.  They are $\Gamma$-invariant.  This suggests the parametrix construction of \cite{HHM2004} (see also \cite{Vaillant}) can be used to obtain a parametrix for $D_{\fob}$ and $D_{\foc}$. 

We must first introduce some notation.   All of the various spaces of
differential forms in \cite{HHM2004} have analogues on spaces
with foliated boundaries, which can be defined by
lifting via $q$.  For example, we define the space
$\Omega_{\fob}^{*}(M)$ to be the space of smooth forms on $M$ which, when
restricted to $\p M\times (0,\frac{\epsilon}{2})$, pull
back via $q_{\wt{\pi}}$ to lie in the space $\Omega_{fb}^{*}(\p \wt{M} \times [0,\epsilon)_{x})$ in Section 5.1 of
\cite{HHM2004}.  The same goes for the spaces of conormal
distributions of degree $k$, $\mathcal{A}^{*}\Omega^{k}_{\fob}(M)$ and
polyhomogeneous forms $\mathcal{A}_{phg}^{*}\Omega^{k}_{\fob}(M)$.  We then
have, for example, the complexes
\begin{equation}
 \cdots \xrightarrow{d} x^{a}\mathcal{A}^*\Omega^{k}(g_{\fob}, U) \xrightarrow{d}
  x^{a+1}\mathcal{A}^*\Omega^{k + 1}(g_{\fob}, U) \xrightarrow{d} \cdots.
\end{equation}
and
\begin{equation}
 \cdots \xrightarrow{d} x^{a}\mathcal{A}^*\Omega^{k}(g_{\foc}, U) \xrightarrow{d}
  x^{a}\mathcal{A}^*\Omega^{k + 1}(g_{\foc}, U) \xrightarrow{d} \cdots,
\end{equation}
and as described in section 2.5 of \cite{HHM2004}, the cohomologies of
these two complexes are naturally isomorphic to the corresponding weighted $L^{2}$
cohomologies in \eqref{eq:localweightedcomplex} and
\eqref{eq:localweightedcomplexfob}.  The same is true for the
polyhomogeneous spaces.

Let $\Pi_0$ be the 
family of projections onto the space of harmonic forms for the family
of Hodge-deRham operators $D_{k}$ associated to the family of
metrics $k$ on the fibres of $\widetilde{\pi}\colon \pa \widetilde{M}\to B$.
Let $\Pi_{\perp}$ be the orthogonal complement of this family.
\begin{definition}
We define $x^{a+1}\Pi_0 L^2\Omega_{\fob}^*(M)\oplus
x^a\Pi_{\perp}L^2\Omega^*_{\fob}(M)$ to be the space of currents on
$M$ which are in $L^2$ when restricted to $M\setminus (\pa M\times
[0,\frac{\epsilon}{2}))$ and are such that their restrictions to
$\pa M\times [0,\epsilon)_x$ pull back under the 
quotient map $q_{\wt{\pi}}$ to give elements in the space $x^{a+1}\Pi_0
L^2\Omega_{fb}^*(\pa\widetilde{M}\times [0,\epsilon)_x)\oplus
x^a\Pi_{\perp}L^2\Omega^*_{fb}(\pa\widetilde{M}\times [0,\epsilon)_x)$
defined in \cite{HHM2004}.  The spaces
\begin{equation}
  \begin{gathered}
    x^{a-1}\Pi_0 H^1\Omega_{\fob}^*(M)\oplus
    x^a\Pi_{\perp}H^1\Omega^*_{\fob}(M), \\
    x^{a}\Pi_0
    L^2\Omega_{\foc}^*(M)\oplus
    x^{a-1}\Pi_{\perp}L^2\Omega^*_{\foc}(M) \mbox{ and } \\ x^{a}\Pi_0
    H^1\Omega_{\foc}^*(M)\oplus
    x^{a+1}\Pi_{\perp}H^1\Omega^*_{\fob}(M)
  \end{gathered}
\end{equation}
are defined in a similar way.   
\label{wtl.1}\end{definition}

In order to use the work in Section 5 of \cite{HHM2004} directly, we
must make one additional assumption, namely that the pullback
$q^*g_{\foc}$ is exact, i.e.\ of the form
\eqref{eq:foliatedcuspmet} from the introduction, and that for fixed
$x = \epsilon$, the fibre metric $k$ annihilates the orthocomplement
of $\mbox{T}\wt{F}\subset (\mbox{T}\p \wt{M} , g_{\foc} \rvert_{x =
  \epsilon})$.  As described in \cite[\S~8]{pomfb}, such a metric is easily constructed.

\begin{proposition}
Suppose that $a$ is not an indicial root for $\Pi_0
x^{-1}\widetilde{D}_{\fob}\Pi_0$.  Then 
\[
   D_{\fob}\colon x^aH^1_{\fob}(M)\to x^{a+1}\Pi_0 L^2\Omega_{\fob}^*(M)\oplus x^a\Pi_{\perp}L^2\Omega^*_{\fob}(M)
\]
and 
\[
  D_{\fob}\colon x^{a-1}\Pi_0 H^1\Omega_{\fob}^*(M)\oplus x^a\Pi_{\perp}H^1\Omega^*_{\fob}(M) \to x^aL^2\Omega^*_{\fob}(M)
\]
are Fredholm.  If $D_{\fob}\omega=0$, then $\omega$ is polyhomogeneous
with exponents in its expansion determined by the indicial roots of
$\Pi_0x^{-1}\widetilde{D}_{\fob}\Pi_0$, while if $\eta\in
\cA^a\Omega^*_{\fob}(M)$, $\zeta \in x^{c-1}\Pi_0
H^1\Omega^*_{\fob}\oplus x^c \Pi_{\perp}H^1\Omega^*_{\fob}(M)$ for $c<a$
and $\eta= D_{\fob}\zeta$, then
$\zeta\in \Pi_0\cA^I_{\phg}\Omega^*_{\fob}(M)+ \cA^a\Omega^*_{\fob}(M)$.
\label{wtl.2}\end{proposition}
\begin{proof}
When the foliation $\cF$ is a fibration, this is just
\cite[Proposition~16]{HHM2004}.  To obtain the result in the foliated
case, we can apply the construction of \cite{HHM2004} to the operator
$\widetilde{D}_{\fob}$ near $\pa\widetilde{M}$ to obtain a parametrix
$\widetilde{G}_{\fob}$.  By averaging with $\Gamma$ if necessary, we
can assume $\widetilde{G}_{\fob}$ is $\Gamma$-invariant.  This means
(see for instance \cite{pomfb}) that $\widetilde{G}_{\fob}$ descends to
an operator $q_*\widetilde{G}_{\fob}$ on $\pa M\times [0,\epsilon)_x$.
This can be combined with 
a parametrix of $D_{\fob}$ in the interior of $M$ to obtain a global
parametrix $G_{\fob}$ of $D_{\fob}$, from which the result follows.  
\end{proof}

Using a similar approach, we also obtained the following generalization of \cite[Proposition~17]{HHM2004}.

\begin{proposition}
Suppose that $a$ is not an indicial root for $\Pi_0 \widetilde{D}_{\foc}\Pi_0$.  Then 
\[
   D_{\foc}\colon x^aH^1_{\foc}(M)\to x^{a}\Pi_0 L^2\Omega_{\foc}^*(M)\oplus x^{a-1}\Pi_{\perp}L^2\Omega^*_{\foc}(M)
\]
is Fredholm.  If $a+1$ is not an indicial root, then
\[
  D_{\foc}\colon x^{a}\Pi_0 H^1\Omega_{\foc}^*(M)\oplus x^{a+1}\Pi_{\perp}H^1\Omega^*_{\foc}(M) \to x^aL^2\Omega^*_{\foc}(M)
\]
is Fredholm.  If $D_{\foc}\omega=0$, then $\omega$ is polyhomogeneous
with exponents in its expansion determined by the indicial roots of
$\Pi_0\widetilde{D}_{\foc}\Pi_0$, while if $\eta\in
\cA^a\Omega^*_{\foc}(M)$, $\zeta \in x^{c}\Pi_0 H^1\Omega^*_{\foc}\oplus
x^{c+1} \Pi_{\perp}H^1\Omega^*_{\foc}(M)$ for $c<a$ and $\eta=
D_{\foc}\zeta$, then
$\zeta\in \Pi_0\cA^I_{\phg}\Omega^*_{\foc}(M)+ \cA^a\Omega^*_{\foc}(M)$.
\label{wtl.3}\end{proposition}

With these results, we can now relate $L^2$-harmonic forms with weighted $L^2$-cohomology.

\begin{theorem}
If $(M,g_{\fob})$ is a manifold with a foliated boundary metric such
that the boundary foliation is a good Seifert fibration (see Definition~\ref{tfb.2b}), then
for every $k$ there is a natural isomorphism
\begin{equation}
   \Psi\colon L^2\mathcal{H}^k(M, g_{\fob})\to
   \Im(\mathcal{W}\!H^k(M,g_{\fob},\epsilon_0)\to \mathcal{W}\!H^k(M,
   g_{\fob},-\epsilon_0))
\label{wtl.4b}\end{equation}
for $\epsilon_0>0$ sufficiently small.
\label{wtl.4}\end{theorem}
\begin{proof}
Thanks to Proposition~\ref{wtl.2}, the proof is a straightforward
adaptation of the proof of \cite[Theorem~1C]{HHM2004}.  First, if
$\omega\in L^2\mathcal{H}^k(M, g_{\fob})$, then by Proposition~\ref{wtl.2}, we
know $\omega$ is polyhomogeneous, so in particular is in
$x^{\epsilon_0}L^2\Omega^k_{\fob}(M)$ for some small $\epsilon_0>0$,
specifically any $\epsilon_{0}$ smaller than the smallest indicial root
of the Hodge-deRham operator.
This gives the mapping \eqref{wtl.4b}.  If $\Psi(\omega)=0$, then
$\omega= d\zeta$ for some $\zeta\in
x^{-1 -\epsilon_0}L^2\Omega^{k-1}_{\fob}(M)$.  As described above, we can
choose $\zeta$ to
be polyhomogeneous.   Then
\begin{equation}
\|\omega\|^2= \int_M d\zeta\wedge * \omega= \int_M d(\zeta\wedge *\omega) =
\lim_{x\to0} \int_{\pa M\times \{x\}}\zeta\wedge *\omega = 
\frac{1}{\absv{\Gamma}} \lim_{x\to0} \int_{\pa M\times
    \{x\}}q^{*}\zeta\wedge * q^{*}\omega.
\label{wtl.5}\end{equation}
The latter goes to zero exactly as shown in \cite{HHM2004}.

To establish surjectivity, we note that thanks to Proposition~\ref{wtl.2}, the space of $L^2$-harmonic forms $L^2\mathcal{H}^*(M)$ can be identified with the cokernel of the map
\[
  D_{\fob}\colon x^{-\epsilon_0-1} \Pi_0 H^1 \Omega^*_{\fob}(M) \oplus x^{-\epsilon_0}\Pi_{\perp} H^1\Omega^*_{\fob}(M)\to x^{-\epsilon_0}L^2\Omega^*_{\fob}(M)
\] 
for $\epsilon_0>0$ small enough.  This means that
$x^{-\epsilon_0}L^2\Omega^*_{\fob}(M)$ is equal to
\begin{equation}
 \Im\left( D_{\fob}\colon x^{-\epsilon_0-1} \Pi_0 H^1 \Omega^*_{\fob}(M) \oplus x^{-\epsilon_0}\Pi_{\perp} H^1\Omega^*_{\fob}(M)\to x^{-\epsilon_0}L^2\Omega^*_{\fob}(M)\right) \oplus L^2\mathcal{H}^*(M).
 \label{wtl.6}\end{equation}
Thus, suppose $\eta\in x^{\epsilon_0}L^2\Omega^k_{\fob}(M)$ is a
representative for a class in the space on the right of
\eqref{wtl.4b}.  Again, we
choose $\eta$ to be polyhomogeneous.  According to \eqref{wtl.6}, the
form $\eta$ can be rewritten as
\[
  \eta= D_{\fob}\zeta+ \gamma, \quad \zeta\in x^{-\epsilon_0-1} \Pi_0 H^1 \Omega^*_{\fob}(M) \oplus x^{-\epsilon_0}\Pi_{\perp} H^1\Omega^*_{\fob}(M),
  \; \gamma\in L^2\mathcal{H}^*(M).
\]
Restricting this identity to the collar neighborhood $\pa M\times [0,\epsilon)_x$ and pulling it back to $\pa\widetilde{M}\times [0,\epsilon)_x$ via the quotient map $q$, we use the same argument as in \cite{HHM2004} to conclude that 
\[
  \zeta= \zeta_0 + \zeta' \; \mbox{with} \; \zeta'\in \cA^{\epsilon_0}\Omega^*_{\fob}(M) \; \mbox{and} \; \zeta_0\in x^{\frac{b-1}{2}} \left( \frac{\alpha_{\frac{b-1}{2}}}{x^{\frac{b-1}{2}}} + \frac{dx}{x^2}\wedge \frac{\beta_{\frac{b+1}{2}}}{x^{\frac{b+1}{2}}} \right),
\] 
where $q^*\alpha_{\frac{b-1}{2}}, q^* \beta_{\frac{b+1}{2}} \in \ker \mathbb{D}$.  Here, $\mathbb{D}$ is the operator defined in \cite{HHM2004} for the metric $q^*g_{\fob}$.  This can be used to show, as in \cite{HHM2004}, that $\|\delta \zeta\|^2=0$ using integration by parts.  Thus, $\delta\zeta=0$ and consequently we have that
\[
    \eta= d\zeta + \gamma.
\]  
This shows that the cohomology class represented by $\eta$ is in the image of the map \eqref{wtl.4b}, establishing surjectivity.  

\end{proof}

Using Proposition~\ref{wtl.3} instead of Proposition~\ref{wtl.2} and adapting the proof of \cite[Theorem~2C]{HHM2004} in a straightforward manner, we obtain also the following.  
\begin{theorem}\label{thm:hodgetoweightedc}
If $(M,g_{\foc})$ is a manifold with a foliated cusp metric such that
the boundary foliation is a good Seifert fibration (see Definition~\ref{tfb.2b}), then for
every $k$ there is a natural isomorphism
\begin{equation}
   L^2\mathcal{H}^k_{\foc}(M)\to \Im(W\!H^k(M,g_{\foc},\epsilon_0)\to W\!H^k(M, g_{\foc},-\epsilon_0))
\label{wtl.7b}\end{equation}
for $\epsilon_0>0$ sufficiently small.
\label{wtl.7}\end{theorem}

\section{Proofs of main theorems}\label{sec:proofs}

Exactly as in \cite[\S 5.5]{HHM2004}, Theorems \ref{thm:main1} and \ref{thm:main2} follow directly from 
Theorem \ref{thm:wcihfol} in light of Theorems
\ref{wtl.4} and \ref{wtl.7}.  For example, for foliated cusp metrics, combining \eqref{wtl.7b} and
\eqref{eq:localequivweighted} gives
\begin{equation}
  \label{eq:7}
   L^2\mathcal{H}^k_{\foc}(M) \simeq
   \Im(I\!H^k_{[\epsilon + f/2]}(M)\to I\!H^k_{[- \epsilon + f/2]}(M) ),
\end{equation}
for $\epsilon > 0$ sufficiently small.
These are exactly the lower and upper middle perversities,
\[
\underline{\mathfrak{m}}(f+1)=\left\{ \begin{array}{ll}
\frac{f-1}{2} & f \ \ \mbox{odd},\\ \frac{f}{2}   & f \ \ \mbox{even},
\end{array} \right.
\qquad
\overline{\mathfrak{m}}(f+1)=\left\{ \begin{array}{ll}
\frac{f-1}{2} & f \ \ \mbox{odd},\\
\frac{f}{2}-1  & f \ \ \mbox{even}.
\end{array} \right. 
\]
The proof of Theorem \ref{thm:main1} is similar using
\eqref{eq:fobequivfoc}.

For Theorem~\ref{Witt.1},
recall that  the Witt condition is simply
the assumption that for any
link $L$ in a regular neighborhood, if $L$ is even dimensional, then the intersection cohomology
\begin{equation}
I\!H^{\dim L / 2}_{\overline{m}}(L) = 0.\label{eq:Witt3}
\end{equation}
In our case, a necessary and sufficient condition to insure that the space $X$ is Witt is to require that
$H^{\frac{f}{2}}(\widetilde{F})=0$ (if $f$ is even).  Indeed, the condition is necessary since, thanks to the fact the base $B$ is chosen to be an effective orbifold, the link of the regular neighborhoods is $\tF$ for most of the points on $B$.  To see that the condition is sufficient, notice that  
if $L$ is even and $H^{\frac{f}{2}}(\widetilde{F})=0$, then \eqref{eq:4} and the map
$I\!H^{\dim L / 2}_{\overline{m}}(L)  \hookrightarrow I\!H^{\dim L /
  2}_{\overline{m}}(\wt{L})$ from \eqref{eq:square} induces an  injective map $I\!H^{\dim L / 2}_{\overline{m}}(L) \hookrightarrow
I\!H^{\dim L / 2}_{\overline{m}}(C_{1}(\wt{F}))$, and  $I\!H^{\dim L / 2}_{\overline{m}}(C_{1}(\wt{F}))$ is easily seen to be $0$ by Proposition \ref{thm:localtoglobal} and the fact $\dim L\ge f$.
Therefore, since either $f$ is odd or $H^{\frac{f}{2}}(\tF)=0$ in the Witt case, we can just apply Theorem~\ref{thm:wcihfol} with $a=0$ to obtain $H^*_{(2)}(M,g_{\foc}) \cong I\!H^*_{\underline{\mathfrak{m}}}(X,B)$.  By Corollary~\ref{dclosed.1}, we also know that the differential in the $L^2$-complex is closed, so that the $L^2$-cohomology is also naturally identified with the space of $L^2$-harmonic forms, completing the proof of Theorem~\ref{Witt.1}.  In particular, notice that this does not require the Seifert fibration $\cF$ to be good.

\section{Some examples}

We will now verify Theorem \ref{thm:main1} directly for a simple example: a quotient of a
multi-Taub-NUT space.  These are asymptotically locally flat (ALF)
gravitational instantons of type $A_{k-1}$
first written down in \cite{GH1978}.  Such a manifold, $\wt{M}$, contains
$k$ points $p_{1}, \dots,
p_{k}$ such that $\wt{M} - \set{ p_{1}, \dots, p_{k}}$  is diffeomorphic to
a principle $S^{1}$ bundle over
$\mathbb{R}^{3} - \set{z_{1}, \dots, z_{k}}$ of degree $-1$ near each
$z_{i}$.  The metric is locally of the form
\begin{equation*}
  g_{ALF}^{k} = V \lp dx_{1}^{2} + dx_{2}^{2} + dx_{3}^{2}\rp +
  V^{-1} \eta^{2},
\end{equation*}
where $\eta$ is a choice of connection 1-form for the principle $S^1$ bundle with curvature given by $ * d V$, and 
\begin{equation*}
  V(z) = 1 + \sum_{i = 1}^{k} \frac{1}{\absv{z - z_{i}}}.
\end{equation*}
We will assume that the $z_{i}$ lie  at the
points $(\cos(2\pi / k), \sin(2\pi / k), 0)$. The map defined by
\begin{equation*}
(x_{1}, x_{2}, x_{3}) \mapsto
(\cos(2\pi / k) x_{1}, \sin(2\pi / k) x_{2}, x_{3})
\end{equation*}
generates an action of $\bbZ_k\subset \SO(3)$ on $\bbR^3$ which acts freely on the monopole points $z_1, \ldots, z_k$.  As explained in \cite[p.106]{Wright2012}, see also \cite[Proposition~2.7]{Honda-Viaclovsky2010}, this action of $\bbZ_k$ can be lifted to an action on $\wt{M}$ by isometries.  The lifted action is not unique, but it is once we specify the action of $\bbZ_k$ on the $\bbS^1$ fiber above the origin in $\bbR^3$, see again \cite[p.106]{Wright2012}.  In order for $\bbZ_k$ to act freely on $\wt{M}$, we will require that the action on the $\bbS^1$ fibre above the origin is the one were the generator $1\in\bbZ_k$ acts by multiplication by $e^{\frac{2\pi i}{k}}$.  

For this free action of $\bbZ_k$ by isometries on $\wt{M}$, we will now verify explicitly that Corollary~\ref{thm:fbsphere}
holds for $$M := \wt{M}/\bbZ_{k}$$ in degree $2$.  That is, we will
show that if $X$ is the orbifold obtained as in \eqref{eq:X} from $\overline{M}$ by collapsing the boundary
foliation over the space of leaves, then
\begin{equation}\label{eq:goal}
L^{2}\mathcal{H}^{2}(M)
\simeq H^{2}(X).
\end{equation}

As pointed out in  in \cite{R1986}, the $2$-forms 
\begin{equation*}
\Omega_i= *dV_i-d\left(\frac{V_i}{V}\eta\right), 
\qquad V_i=\frac{2m}{|x-p_i|},  \quad i\in\{1,\ldots,k \},
\end{equation*}
give a basis of  the
space $L^{2}\mathcal{H}^{2}(\wt{M})$.   
The space $L^{2}\mathcal{H}^{2}(M)$ is isomorphic to the space
of $\bbZ_{k}$-invariant $L^{2}$ harmonic forms on $\wt{M}$, which is the
one-dimensional space generated by $\sum_{i = 1}^{k}\Omega_{i}$.  Thus
$L^{2}\mathcal{H}^{2}(M) \simeq \mathbb{R}$.

Now we want to compute $H^{2}(X)$.    As described in \cite{H1979}, the inverse
image via $\wt{M} - \set{p_{1}, \dots, p_{k}}\lra \mathbb{R}^{3} -
\set{z_{1}, \dots, z_{k}}$ of the straight line connecting $z_{i}$ to
$z_{j}$ is a nontrivial class in $H_{2}(\wt{M}) \simeq \mathbb{R}^{k -
  1}$.  Let $\gamma_{i}, i = 1, \dots, k - 1,$ denote the Poincar\'e
dual of the class
corresponding to the line connecting $z_{i}$ to $z_{i + 1}$, and
$\gamma_{k}$ that for the line from $z_{k}$ to $z_{1}$.  Then
$\gamma_{1}, \dots, \gamma_{k - 1}$ form a basis for $H^{2}(\wt{M})$, and,
if $k > 2$ the
intersections numbers are
\begin{align*}
  \gamma_{i}^{2} &= -2, \\
  \gamma_{i}\cdot\gamma_{j} &= 1 \mbox{ if } i - j \equiv \pm 1 \mbox{ (mod
    $k$),
  }  \\
  \gamma_{i}\cdot\gamma_{j} &= 0 \mbox{ otherwise. }
\end{align*}
From this, one can easily check that the intersection pairing is
negative definite on $H^{2}(\wt{M})$ and that
\begin{equation}
\sum_{i = 1}^{k}\gamma_{i} =
 0.\label{eq:cobound}
\end{equation}
in cohomology.  If $r$ is the radial coordinate on $\mathbb{R}^{3}$,
then the set $\set{r > R}$ for sufficiently large $R$ is diffeomorphic
to the total space of the fibre bundle 
\begin{equation}
\xymatrix{
S^{1} \ar[r] &\pa \wt{M} \ar[d]\\  
  & S^{2},
  }
\label{fb.1}\end{equation}
 of degree
$k$ crossed with $(R, \infty)_{r}$.  Let $\overline{\wt{M}}$ be the manifold with boundary obtained from $\wt{M}$ by gluing the boundary $\pa \wt{M}\cong \bbS^3 / \bbZ_k$.  
Let $\wt{X}$ the space obtained from $\overline{\wt{M}}$ by collapsing the fibres of the bundle \eqref{fb.1} on the boundary of $\overline{\wt{M}}$ onto the base $S^2$.  Since the fibres of \eqref{fb.1} are circles, the space $\wt{X}$ is naturally a smooth manifold and  $\wt{X} - \wt{M} \simeq S^{2}$.  We claim that the de Rham cohomology of $\wt{X}$ satisfies $H^{2}_{dR}(\wt{X}) \simeq
\mathbb{R}^{k}$ where a basis is given by the 
$\gamma_{i}, i = 1, \dots, k - 1,$ and the Poincar\'e dual of $S^2= \wt{X}\setminus \wt{M}\subset \wt{X}$, call
it $\overline{\gamma}$.  Indeed, we use the Mayer-Vietoris sequence for
the open cover $U = \set{r > R}$ and $V = \set{r < R + \epsilon}$.
The set $U \cap V$ deformation retracts onto $\pa\wt{M}\cong S^{3}/\bbZ_{k}$, which has
trivial cohomology in degrees 1 and 2. Thus
\begin{equation*}
  H_{dR}^{2}(\wt{X}) \simeq H_{dR}^{2}(U) \oplus H_{dR}^{2}(V).
\end{equation*}
Since $V$ is diffeomorphic to $\wt{M}$ (and thus has the same homology),
and since $U$ deformation retracts onto $S^{2}$, this proves our claim.
 The de Rham cohomology of the orbifold $X = \wt{X}/\bbZ_{k}$ is
identified by definition with the $\bbZ_k$-invariant part of the de Rham cohomology of
$\wt{X}$, and it is isomorphic to the singular cohomology of $X$,
modulo torsion.  By
\eqref{eq:cobound}, the $\gamma_{i}$'s have no $\bbZ_k$-invariant part.  On the other hand, $\overline{\gamma}$ is $\bbZ_k$-invariant, so the space of $\bbZ_k$-invariant de Rham classes is one dimensional, and thus the singular cohomology of
$X$ is one dimensional and \eqref{eq:goal} holds as claimed.  

Of course, our results become really interesting when $M$ is not a global quotient of a manifold with fibred boundary.  This is not hard to construct as the next example shows.     
\begin{example}
Let $\cF$ be a foliation of  a closed smooth oriented $3$-manifold $Y$ by circles corresponding to a good Seifert fibration on a closed surface orbifold $\Sigma$.  As discussed in \cite[Example~3.5]{pomfb}, the Seifert fibration is automatically good provided $\Sigma$ is not a tear drop or a $2$-sphere with two cone points having different cone angles.  We suppose of course that $\cF$ is not a foliation coming from a fibred bundle.  In this case, from 
Definition~\ref{thm:goodassumption}, we know that  $Y= \widetilde{Y}/ \Gamma$ with $\widetilde{Y}$ the total space of a circle bundle over $\widetilde{\Sigma}$ on which $\Gamma$ acts equivariantly.  Moreover, the leaves of $\cF$ are then given by the image of the fibres of the circle bundle over $\widetilde{\Sigma}$ under the quotient map.

Let $\overline{M}$ be a compact oriented manifold with boundary such
that $\pa\overline{M}=Y$.  Notice that since the oriented cobordism
group is trivial in dimension $3$ by \cite{Wall1960}, such a
$\overline{M}$ always exists.  We want to choose $\overline{M}$ such
that the $\Gamma$-bundle $\widetilde{Y}\to Y$ does not extend to a
$\Gamma$-bundle $\widetilde{M}\to M$, where $M= \overline{M}\setminus
\pa\overline{M}$.  In other words, if $\nu: Y\to B\Gamma$ is the
corresponding map in the classifying space $B\Gamma$, we want to
choose $\overline{M}$ such that $\nu$ does not extend to a map from
$\overline{M}$ to $B\Gamma$.   To make such a ch
oice, pick a loop $\gamma: \bbS^1 \to Y$ such that $\nu\circ \gamma$ is a non-trivial element of $\pi_1(B\Gamma)$.  Such a loop exists since we suppose the $\Gamma$-bundle $\widetilde{Y}\to Y$ is not trivial.  

Clearly then, it will be impossible to extend the map $\nu$ to
$\overline{M}$ if the loop $\gamma$ is contractible in $\overline{M}$.
If this is not the case, move the loop $\gamma$ smoothly in the
interior of $\overline{M}$ to an  embedded loop $\gamma': \bbS^1 \to
M$.   Since $Y$ is of dimension $3$ and orientable, its tangent bundle is trivial.  Consequently, the normal bundle of $\gamma(S^1)$ in $Y$ and of $\gamma'(\bbS^1)$ in $M$ are trivial.  We can thus perform a surgery of codimension 3 (see for instance \cite[p.99]{Lawson-Michelsohn}) along $\gamma'(\bbS^1)$ in $\overline{M}$ to obtain a new oriented manifold with boundary $\overline{M}'$ such that $\pa \overline{M}'=Y$ and $\gamma$ is contractible in $\overline{M}'$.  In particular, the map $\nu$ does not extend to $\overline{M}'$.  Thus, $\overline{M}'$ gives an example of a manifold with boundary foliated by good Seifert fibration for which the corresponding $\Gamma$-bundle on $\pa \overline{M}'$ cannot be extended to $\overline{M}'$.  
\label{ngq.1}\end{example}

%
%
%
%
%
\providecommand{\bysame}{\leavevmode\hbox to3em{\hrulefill}\thinspace}
\providecommand{\MR}{\relax\ifhmode\unskip\space\fi MR }
\providecommand{\MRhref}[2]{%
  \href{http://www.ams.org/mathscinet-getitem?mr=#1}{#2}
}
\providecommand{\href}[2]{#2}


\begin{thebibliography}{10}

\bibitem{ALMP2011}
P.~Albin, E.~Leichtnam, R.~Mazzeo, and P.~Piazza, \emph{The signature package
  on witt spaces}, available online at arXiv:1112.0989, 2011.

\bibitem{APSI}
M.~F. Atiyah, V.~K. Patodi, and I.~M. Singer, \emph{Spectral asymmetry and
  {R}iemannian geometry. {I}}, Math. Proc. Cambridge Philos. Soc. \textbf{77}
  (1975), 43--69.

\bibitem{B1983}
A.~Borel and et~al., \emph{Intersection cohomology}, Modern Birkh\"auser
  Classics, Birkh\"auser Boston Inc., Boston, MA, 2008, Notes on the seminar
  held at the University of Bern, Bern, 1983, Reprint of the 1984 edition.

\bibitem{BHS1992}
Jean-Paul Brasselet, Gilbert Hector, and Martin Saralegi,
  \emph{{${\mathcal{L}}^2$}-cohomologie des espaces stratifi\'es}, Manuscripta
  Math. \textbf{76} (1992), no.~1, 21--32. \MR{1171153 (93i:58009)}

\bibitem{C1980}
Jeff Cheeger, \emph{On the {H}odge theory of {R}iemannian pseudomanifolds},
  Geometry of the {L}aplace operator ({P}roc. {S}ympos. {P}ure {M}ath., {U}niv.
  {H}awaii, {H}onolulu, {H}awaii, 1979), Proc. Sympos. Pure Math., XXXVI, Amer.
  Math. Soc., Providence, R.I., 1980, pp.~91--146.

\bibitem{CGM1982}
Jeff Cheeger, Mark Goresky, and Robert MacPherson, \emph{{$L^{2}$}-cohomology
  and intersection homology of singular algebraic varieties}, Seminar on
  {D}ifferential {G}eometry, Ann. of Math. Stud., vol. 102, Princeton Univ.
  Press, Princeton, N.J., 1982, pp.~303--340.

\bibitem{Chen-Ruan2004}
Weimin Chen and Yongbin Ruan, \emph{A new cohomology theory of orbifold}, Comm.
  Math. Phys. \textbf{248} (2004), no.~1, 1--31. \MR{2104605 (2005j:57036)}

\bibitem{DLR2011}
C.~Debord, J.-M. Lescure, and F.~Rochon, \emph{Pseudodifferential operators on
  manifolds with fibred corners}, available online at arXiv:1112.4575, 2011.

\bibitem{DK2000}
J.~J. Duistermaat and J.~A.~C. Kolk, \emph{Lie groups}, Universitext,
  Springer-Verlag, Berlin, 2000.

\bibitem{GH1978}
G.~W. Gibbons and S.~W. Hawking, \emph{Gravitational multi-instantons}, Phys.
  Lett. B \textbf{78} (1978), no.~4, 430--432.

\bibitem{Goette2011}
Sebastian Goette, \emph{Adiabatic limits of seifert fibrations, dedekind sums,
  and the diffeomorphism type of certain $7$-manifolds}, available online at
  arXiv:1108.5614, 2011.

\bibitem{GM1980}
Mark Goresky and Robert MacPherson, \emph{Intersection homology theory},
  Topology \textbf{19} (1980), no.~2, 135--162.

\bibitem{GM1983}
\bysame, \emph{Intersection homology. {II}}, Invent. Math. \textbf{72} (1983),
  no.~1, 77--129.

\bibitem{HHM2004}
Tam{\'a}s Hausel, Eugenie Hunsicker, and Rafe Mazzeo, \emph{Hodge cohomology of
  gravitational instantons}, Duke Math. J. \textbf{122} (2004), no.~3,
  485--548.

\bibitem{H1979}
N.~J. Hitchin, \emph{Polygons and gravitons}, Math. Proc. Cambridge Philos.
  Soc. \textbf{85} (1979), no.~3, 465--476.

\bibitem{Honda-Viaclovsky2010}
N.~Honda and J.~Viaclovsky, \emph{Conformal symmetries of self-dual hyperbolic
  monopole metrics}, available online at arXiv:1002.2119, 2010.

\bibitem{Lawson-Michelsohn}
H.B. Lawson and M-L Michelsohn, \emph{Spin geometry}, Princeton Univ. Press,
  Princeton, 1989.

\bibitem{Ma1988}
Rafe Mazzeo, \emph{The {H}odge cohomology of a conformally compact metric}, J.
  Differential Geom. \textbf{28} (1988), no.~2, 309--339.

\bibitem{MP1990}
Rafe Mazzeo and Ralph~S. Phillips, \emph{Hodge theory on hyperbolic manifolds},
  Duke Math. J. \textbf{60} (1990), no.~2, 509--559.

\bibitem{Molino1988}
P.~Molino, \emph{Riemannian foliations}, Birkh\"auser, Boston, 1988.

\bibitem{Nagase1987}
Masayoshi Nagase, \emph{Sheaf theoretic {$L^2$}-cohomology}, Complex analytic
  singularities, Adv. Stud. Pure Math., vol.~8, North-Holland, Amsterdam, 1987,
  pp.~273--279. \MR{894298 (88g:58009)}

\bibitem{N1999}
Arvind Nair, \emph{Weighted cohomology of arithmetic groups}, Ann. of Math. (2)
  \textbf{150} (1999), no.~1, 1--31.

\bibitem{Pflaum2001}
M.~Pflaum, \emph{Analytic and geometric study of stratified spaces}, Lecture
  notes in mathematics, Springer-Verlag, Berlin, 2001.

\bibitem{pomfb}
Fr\'ed\'eric Rochon, \emph{Pseudodidfferential operators on manifolds with
  foliated boundaries}, J. Funct. Anal. \textbf{262} (2012), 1309--1362.

\bibitem{R1986}
P.~J. Ruback, \emph{The motion of {K}aluza-{K}lein monopoles}, Comm. Math.
  Phys. \textbf{107} (1986), no.~1, 93--102.

\bibitem{SS1990}
Leslie Saper and Mark Stern, \emph{{$L_2$}-cohomology of arithmetic varieties},
  Ann. of Math. (2) \textbf{132} (1990), no.~1, 1--69.

\bibitem{Thurston}
W.~Thurston, \emph{The geometry and topology of three-manifolds}, electronic
  edition of the 1980 notes distributed by Princeton University.

\bibitem{Vaillant}
Boris Vaillant, \emph{Index and spectral theory for manifolds with generalized
  fibred cusps}, available online at arXiv: math/0102072v1, 2001.

\bibitem{Wall1960}
C.~T.~C. Wall, \emph{Determination of the cobordism ring}, Ann. of Math. (2)
  \textbf{72} (1960), 292--311. \MR{0120654 (22 \#11403)}

\bibitem{Wright2012}
Evan~P. Wright, \emph{Quotients of gravitational instantons}, Ann. Global Anal.
  Geom. \textbf{41} (2012), no.~1, 91--108. \MR{2860398}

\bibitem{Z1982}
Steven Zucker, \emph{{$L_{2}$} cohomology of warped products and arithmetic
  groups}, Invent. Math. \textbf{70} (1982/83), no.~2, 169--218.

\end{thebibliography}
\end{document}